\documentclass[11pt,a4paper,reqno]{amsart}
\usepackage[utf8]{inputenc}
\usepackage{amsfonts, amsmath, amssymb, amsthm}
\usepackage{dsfont}
\usepackage[margin=1in]{geometry}
\usepackage{comment}
\usepackage{graphicx}
\usepackage{tikz, pgffor, pgfplots, mathtools}
\usepackage{enumerate,lineno,setspace,float}
\usetikzlibrary{arrows.meta} 
\usepackage{hyperref}
\usepackage{color}
\usepackage[numbers, sort&compress]{natbib}
\usepackage[capitalize]{cleveref}
\usepackage{doi}

\usepackage[linesnumbered,ruled]{algorithm2e}

\usepackage{array}

\usepackage{booktabs}

\newcolumntype{L}[1]{>{\raggedright\arraybackslash}p{#1}}

\hypersetup{
    colorlinks=true,       
    linkcolor={red!50!black},        
    citecolor={green!50!black},        
    urlcolor={blue!50!black}
}
\newcommand{\Prob}{\mathds{P}}
\newcommand{\Exp}{\mathbb{E}}
\newcommand{\Href}[1]{\hyperref[#1]{\Cref{#1}}}
\renewcommand{\href}[1]{\hyperref[#1]{\ref{#1}}}
\renewcommand{\eqref}[1]{\hyperref[#1]{(\ref{#1})}}

\newtheorem{theorem}{Theorem}
\newtheorem{lemma}[theorem]{Lemma}
\newtheorem{proposition}[theorem]{Proposition}
\newtheorem{corollary}[theorem]{Corollary}
\newtheorem{definition}[theorem]{Definition}

\newtheorem{claim}[theorem]{Claim}

\newtheorem{fact}[theorem]{Fact}

\numberwithin{theorem}{section}
\newtheorem*{rep@theorem}{\rep@title}
\newcommand{\newreptheorem}[2]{%
\newenvironment{rep#1}[1]{%
 \def\rep@title{#2 \ref{##1}}%
 \begin{rep@theorem}}%
 {\end{rep@theorem}}}
\makeatother

\newreptheorem{theorem}{Theorem}
\newreptheorem{lemma}{Lemma}
\newreptheorem{claim}{Claim}
\newreptheorem{proposition}{Proposition}

\usepackage{color}
\definecolor{red}{rgb}{1,0,0}
\definecolor{Gold}{rgb}{1,0.843,0.2}
\definecolor{DarkBlue}{rgb}{0,0,0.6}

\newcommand{\gold}[1]{{\color{Gold}{#1}}}
\newcommand{\darkBlue}[1]{{\color{DarkBlue}{#1}}}
\newcommand{\mason}[1]{\darkBlue{\bf [Mason: #1]}}

\newcommand{\marcelo}[1]{\gold{\bf [Marcelo: #1]}}
\newcommand{\orarrow}{\overrightarrow}

\DeclareMathOperator{\Bin}{Bin}

\DeclareMathOperator{\NP}{NP}

\def\cA{\ensuremath{\mathcal{A}}}

\def\cC{\ensuremath{\mathcal{C}}}
\def\EE{\ensuremath{\mathbb{E}}}
\def\cE{\ensuremath{\mathcal{E}}}

\def\PP{\ensuremath{\mathbb{P}}}

\def\cF{\ensuremath{\mathcal{F}}}

\def\cM{\ensuremath{\mathcal{M}}}

\let\epsilon=\varepsilon

\pgfplotsset{compat=1.18}

\begin{document}

\onehalfspace
\footskip=28pt

\title{Covering Random Digraphs with Hamilton Cycles}

\author{Asaf Ferber }
\author{Marcelo Sales}
\author{Mason Shurman}

\thanks{The authors are supported by an Air force grant FA9550-23-1-0298.}
\thanks{In addition, A.F is also supported by NSF grant DMS-1953799, NSF Career DMS-2146406, and a Sloan's fellowship. }

\address{Department of Mathematics, University of California, 
    Irvine, CA, USA}
\email{\{asaff|mtsales|mshurman\}@uci.edu}

\begin{abstract}
A covering of a digraph $D$ by Hamilton cycles is a collection of directed Hamilton cycles (not necessarily edge-disjoint) that together cover all the edges of $D$. 
We prove that for $1/2 \geq p\geq \frac{\log^{20} n}{n}$, the random digraph $D_{n,p}$ typically admits an optimal Hamilton cycle covering. Specifically, the edges of $D_{n,p}$ can be covered by a family of $t$ Hamilton cycles, where $t$ is the maximum of the the in-degree and out-degree of the vertices in $D_{n,p}$. Notably, $t$ is the best possible bound, and our assumption on $p$ is optimal up to a polylogarithmic factor.
\end{abstract}

\maketitle

\section{Introduction}

A \emph{Hamilton cycle} in a graph is a cycle that passes through all the vertices of the graph exactly once. A graph is \emph{Hamiltonian} if it contains a Hamilton cycle. Hamiltonicity is one of the most central notions in graph theory, and has been extensively studied in the last decades (see the surveys \cite{gould14survey, kuhnosthus14}). In particular, the seemingly easy problem of deciding if a graph is Hamiltonian is known to be $\NP$-hard and is one of Karp's list of 21 $\NP$-hard problems \cite{karp72}. Thus, there is a lot of interest in finding general sufficient conditions for a graph to be Hamiltonian. For instance, a classical result of Dirac \cite{dirac52} states that every graph on $n$ vertices with minimum degree $n/2$ contains a Hamilton cycle.

Once Hamiltonicity has been established, there are many natural further questions, like the problems of covering and packing a graph by Hamilton cycles. Given graphs $H$ and $G$, an $H$-\emph{packing} of $G$ is a set of edge disjoint copies of $H$ in $G$ and an $H$-\emph{covering} of $G$ is a set of (not necessarily edge-disjoint) copies of $H$ covering all the edges of $G$. An $H$-packing is \emph{optimal} if it is of largest size, while an $H$-covering is called \emph{optimal} if it is of smallest possible size.  

The size of an optimal packing of Hamilton cycles in a graph $G$ is clearly at most~$\lfloor \delta(G)/2 \rfloor$, where $\delta(G)$ denotes the minimum degree of $G$. The problem of determining for which graphs this bound is actually tight is of great interest. The first result in this direction was by Walecki in 1890, who showed that the complete graph $K_n$ can be decomposed into Hamilton cycles when $n$ is odd, and into Hamilton cyles plus a matching when $n$ is even (see \cite{walecki08}). The result was later generalized by Csaba, K\"{u}hn, Lo, Osthus, and Treglown \cite{csaba2016proof}, who showed that one can replace $K_n$ by any regular graph with degree at least $n/2$, settling a longstanding conjecture of Nash and Williams~\cite{nash1970hamiltonian}. This is clearly optimal, since there exist $(n/2-1)$-regular graphs with no Hamilton cycles. Therefore, if we wish to consider sparser graphs we need a different condition than a minimum degree. 

A natural family of graphs to consider is the family of binomial \emph{random graphs}. Let $G_{n,p}$ be the binomial random graph on $n$ vertices with each edge chosen independently with probability $p$. We say that a graph $G:= G_{n,p}$ satisfies a property \emph{with high probability} (whp) if the probability that $G$ satisfies the property tends to $1$ as $n\rightarrow\infty$. The study of packing of Hamilton cycles in random graphs was initiated by Bollob\'{a}s and Frieze~\cite{bollofrieze83matchings} who proved that for every fixed integer $k$, whp $G_{n,p}$ contains $\lfloor k/2 \rfloor$ edge-disjoint Hamilton cycles plus a disjoint perfect matching if $k$ is odd. Recently, the problem of finding optimal Hamilton packings and covering of random graphs has received a great amount of attention, resulting in its complete solution in a series of papers by several authors \cite{bollofrieze83matchings, knox2015edge, krivelevich2012optimal, kuhn2014hamilton}. More precisely, they showed that for any $p$, in $G:= G_{n,p}$ whp there exists a packing of $G$ with $\lfloor \delta(G)/2 \rfloor$ Hamilton cycles, therefore confirming a conjecture of Frieze and Krivelevich \cite{friezekrivelevich08}.
 
The problem of covering the edges of a random graph with Hamilton cycles was first investigated by Glebov, Krivelevich and Szab\'{o} \cite{glebov2014covering}. Note that an optimal Hamilton covering has size at least $\lceil \Delta(G)/2 \rceil$. The authors showed that for $p\geq n^{-1+\epsilon}$ this bound is approximately tight. More precisely, they showed that in this range the edges of $G:= G_{n,p}$ whp can be covered by $(1+o(1))\Delta(G)/2$ Hamilton cycles and conjectured that $\lceil \Delta(G)/2\rceil$ should be the correct size. 
In 2014, Hefetz, K\"{u}hn, Laplinskas and Osthus confirmed this conjecture for $1-n^{1/8} \geq p \geq \log^{117}n/n$ \cite{hefetz2014optimal}, and recently in an impressive work by Dragani\'{c}, Glock, Correia, and Sudakov, the problem was almost entirely solved. They confirmed the conjecture for $1-n^{1/8} \geq p \geq C \log n/n$ for sufficiently large $C$. Up to the constant $C$, this is optimal, since for lower $p$, $G_{n,p}$ is disconnected. 


Our main result concerns an extension of the previous problem to the directed setting. Recall that a \emph{directed graph} (or a \emph{digraph} for short) is a pair $D=(V,E)$ with a set of vertices $V$ and a set of ordered pairs $E\subseteq V^{(2)}$ called \emph{directed edges} of $D$. As a matter of convenience, we shall denote an edge $(x,y)\in E$ by $\orarrow{xy} \in E$. A \emph{Hamilton cycle} in a digraph is a cycle passing through all the vertices exactly once, where all the edges are oriented in the same direction in a cyclic order. Given a digraph $D$, we denote by $\delta^{+}(D)$ and $\delta^{-}(D)$ the minimum out-degree and in-degree, respectively. Similarly, one can define $\Delta^{+}(D)$ and $\Delta^{-}(D)$ the maximum out-degree and in-degree, respectively. Lastly, it would be convenient to define $\delta(D):=\min\{\delta^+(D),\delta^{-}(D)\}$ and $\Delta(D):=\max\{\Delta^+(D),\Delta^-(D)\}.$

Since the above mentioned results regarding the covering problem heavily rely on the so-called P\'{o}sa rotation-extension technique, and since this technique is not applicable in the directed settings, the study of Hamiltonicity problems for directed graphs is considered more difficult. Let $D_{n,p}$ be the binomial random digraph on $n$ vertices, where each directed edge is chosen independently with probability $p$. It is known that a typical $D:= D_{n,p}$ is Hamiltonian assuming $p\geq \frac{\log n+\omega(1)}{n}$ (see \cite{frieze88,mcdiarmid80}). Therefore, it is reasonable to ask about the size of optimal packing and coverings of the random directed graph by Hamilton cycles. A simple argument shows that every Hamilton packing of a digraph $D$ has size at most $\delta(D)$. Moreover, every Hamilton covering of a digraph $D$ has size at least $\Delta(D)$. It was shown in \cite{ferber2017packing} that for $p\geq \log^C n/n$, the digraph $D:= D_{n,p}$ whp has a packing of size $(1-o(1))pn$ of edge-disjoint Hamilton cycles. Moreover, it was also shown that for $p\geq \log^C n/n$, the digraph $D:= D_{n,p}$ whp has a covering of size $(1+o(1))pn$ by Hamilton cycles. Both results are asymptotic optimal, since in a random digraph the in-degree and out-degree concentrates around $p(n-1)$. 

In our main result, we improve upon these bounds and show that for $p \geq \log^{C} n/n$, the trivial bound is indeed optimal for the covering problem in random digraphs.

\begin{theorem}\label{thm:main}
Let $1/2 \geq p\geq\frac{ \log^{20} n}{n}$. Then a digraph $D:=D_{n,p}$ whp can be covered by $\Delta(D)$ directed Hamilton cycles.
\end{theorem}

The paper is organized as follows: In Section \ref{sec:preliminaries}, we give a brief overview of the problem and introduce the standard probabilistic and graph theoretic tools used throughout the paper. Section \ref{sec:CoveringBnp} is devoted to study the problem of finding coverings of the random bipartite graph by perfect matchings. We will use those coverings in Section \ref{sec:coverforest} to show that the random digraph can be almost covered by linear forests. In Section \ref{sec:coverhamilton}, we discuss how to cover linear forests with Hamilton cycles. Finally, a proof of Theorem \ref{thm:main} is given in Section \ref{sec:main}.

\section{Overview and preliminaries}\label{sec:preliminaries}

\subsection{Proof overview}

In this subsection, we will describe our proof strategy to obtain a covering of $D:=D_{n,p}$ with $\Delta(D)$ Hamilton cycles. The proof is split into three parts. Given $D$, let $x_\star$ be the vertex of maximum degree in $D$ and suppose without loss of generality that the maximum degree is an out-degree. In the first part of our proof, we are going to reserve a set of $deg^+(x_\star)-t$ out-edges from $x_\star$, where $t$ is chosen such that $deg^+(x_\star)-t$ is roughly $\sqrt{pn}$. We call this set $E$. We will then proceed to prove that the edges of $D\setminus E$ can be covered by a family of linear forests $\cF=\{F_1,\ldots,F_t\}$ and by a sparse graph $R_1$ with $\Delta(R_1)=o\left(|E|\right)$. In the second part of the proof, we will show that we can cover the edges of the family $\cF$ by $t$ Hamilton cycles $\cC=\{C_1,\ldots,C_t\}$ and another sparse graph $R_2$ with $\Delta(R_2)=o\left(|E|\right)$. Finally, in the last part we will show that one can cover $R:=R_1\cup R_2$ and $E$ by using $|E|$ Hamilton cycles, each of which contains exactly one edge from $E$. This will give us a covering of $D$ by $(\deg^+(x_\star)-t)+t=\Delta(D)$ Hamilton cycles, as desired.

The main idea for the first part is to consider an alternative random model for $D_{n,p}$ using the random bipartite graph model, referred to as $B_{n,p}$, and a random permutation (see \Cref{def:model}). This reduces the problem of finding linear forests in $D_{n,p}$ to finding almost perfect matchings in $B_{n,p}$. In Section \ref{sec:CoveringBnp}, we discuss how to obtain a covering of $B_{n,p}$ by perfect matchings, while in Section \ref{sec:coverforest} we construct our linear forests from such matchings. The main tool for the second and third part of the proof is a result from \cite{ferber2017robust} which states that certain pseudorandom digraphs are Hamiltonian. In Section \ref{sec:coverhamilton}, we will show how each of the linear forests we obtained can be extended into a Hamilton cycle by constructing an auxiliary pseudorandom digraph where we ``contract'' the paths of our linear forest into single vertices (see \Cref{def:auxgraph}).

This paper investigates asymptotic results, so every calculation we show regarding graphs with $n$ vertices will be under the assumption that $n$ is sufficiently large.
We now introduce some preliminary results for our proof.

\subsection{Chernoff's bounds}

We extensively use the following well-known Chernoff bounds for the upper and lower tails of binomial random variables (see \cite{JRLrandomgraphs}).
\begin{theorem}[Chernoff Bound]\label{chernoff}
Suppose $X_1, \ldots X_n$ are independent $0/1$-random variables. Let $X=\sum_{i=1}^n X_i$ and let $\mathbb{E}(X)= \mu$. Then, for $\delta \in (0,1)$, we have
\begin{align*}
    \Prob(|X-\mu| > \delta\mu)\leq 2e^{-\delta^2\mu/3}
\end{align*}
\end{theorem}

For larger deviations, we state another useful version that can be also found in \cite{JRLrandomgraphs}.

\begin{theorem}[\cite{JRLrandomgraphs}, Corollary 2.4]\label{chernoff2}
Suppose $X_1, \ldots X_n$ are independent $0/1$-random variables. Let $X=\sum_{i=1}^n X_i$ and let $a\geq 7\mathbb{E}(X)$ be a real number. Then,
\begin{align*}
\PP(X>a)\leq e^{-a}.
\end{align*}
\end{theorem}

\subsection{Graph theoretical results}


Here, we compile a list of graph theoretical results related to matchings. Given a multigraph $G$, a \emph{proper edge coloring} is a coloring of the edges of $G$ such that no adjacent edge has the same color. In other words, a proper coloring is a partition of $G$ such that each color class is a matching. A multigraph $G$ is said to have multiplicity $\mu:=\mu(G)$ if $\mu$ is the largest integer such that there exists two vertices of $G$ with $\mu$ edges connecting them. The next result is a generalization of the celebrated Vizing's theorem.

\begin{theorem}[\cite{bergefourniervizing}]\label{thm:vizing}
Let $G$ be an undirected multigraph without self-loops. Then $G$ can be properly edge colored with at most $\Delta(G)+\mu(G)$ colors.
\end{theorem}

The next series of results are about finding matching factors in bipartite graphs. We start with Hall's famous marriage theorem (see e.g. Theorem 3.1.11 in \cite{west}). 

\begin{theorem}[Hall's Theorem]\label{thm:halls}
    Let $B$ be a bipartite graph with parts $X$ and $Y$ of the same size. There is a perfect matching in $B$ if and only if for every subset $W\subseteq X$, we have that $|N(W)|\geq |W|$. 
\end{theorem}
It follows from \Cref{thm:halls} that any regular bipartite graph contains a perfect matching 
, and when we remove a perfect matching from a regular graph, the resulting graph is also regular. Thus, by repeatedly applying \Cref{thm:halls} to find perfect matchings, we obtain:
\begin{corollary}
    \label{thm:hallscor}
    The edges of a $d$-regular bipartite graph can be decomposed into $d$ edge-disjoint perfect matchings.
\end{corollary}

We will also need the following lemma due to Gale and Ryser (for instance, see \cite{lovasz2007combinatorial}).

\begin{lemma}\label{lem:gale-ryser}
    Let $B$ be a bipartite graphs with parts $X$ and $Y$ of size $n$. Then $B$ contains an $r$-regular subgraph if and only if for all $S\subseteq X$, $T \subseteq Y$, $$e(S,T)>r(|S|+|T|-n)$$
\end{lemma}

 Let $G=(V,E)$ be any graph and let $f: V\rightarrow \mathbb{N}$. We say that $G$ contains an $f$-\emph{factor} if and only if there exists a spanning subgraph $H\subseteq G$ such that for every $v\in V$ we have $\deg_H(v)=f(v)$. The following lemma due to Ore extends \cref{lem:gale-ryser} and gives a characterization of $f$-factors in bipartite graphs.
 
\begin{lemma}[\cite{OreStudies63}] \label{orestheorem}
Let $B=(X \cup Y,E)$ be a bipartite graph. Let $f: V \to \mathbb{N}$ be such that $\sum_{x \in X}f(x)= \sum_{y \in Y}f(y)$. Then $B$ contains an $f$-factor if and only if for all subsets $S \subseteq X$, $T \subseteq Y$,

\begin{align}\label{eq:ore}
    e(S,T) + \sum_{y \in  Y\setminus T}f(y) \geq \sum_{x \in S}f(x)
\end{align} 
\end{lemma}

\subsection{Properties of Random Graphs}
We now introduce properties of typical random graphs that will be used in the proof of our main result. From now on, when we say that $B$ is a balanced bipartite graph, we name the parts $X$ and $Y$, two disjoint copies of $[n]$. The edge $xy$ refers to the edge between the  $x$th vertex in part $X$, and $y$th vertex in part $Y$. Even though the graph is undirected, in our convention, $xy \neq yx$. For convenience, we make the following technical definition:
\begin{definition}\label{def:good}
    Let $\epsilon>0$ and $p\in (0,1)$. We say that a balanced bipartite graph $B$ is $\left(\varepsilon,p\right)$-good if  
    $$e(S,T)\in \left(1\pm \epsilon\right)stp$$
    for all $S\subseteq X$ and $T\subseteq Y$ with $|S|=s$ and $|T|=t$ such that
\begin{enumerate}
\item[$(a)$]  $s,t\geq \frac{n}{\log^{2/3}n}$, or
\item[$(b)$]  $s \leq \frac{n}{\log^{2/3}n}$ and 
$s\sqrt{\log n} + t  \geq n$, or
\item[$(c)$]  $t \leq \frac{n}{\log^{2/3}n}$ and 
$t\sqrt{\log n} + s  \geq n$.
\end{enumerate}
\end{definition}


For a graph $B$, let $\Delta_2(B)$ be the second largest degree of a vertex in $B$. The following lemma introduces some key properties of typical bipartite graphs.
\begin{lemma}\label{typicalbipartite}
For $1/2 \geq p \geq \frac{\log^{10} n}{n}$, a random bipartite graph $B:= B_{n,p}$ whp satisfies the following properties: 

\begin{enumerate} [$(B1)$]
\setcounter{enumi}{0}
\item $B$ is $\left(\frac{1}{\log^2n},p\right)$-good.
\item $\Delta(B)-\delta(B) < 4 \sqrt{np\log n}$, and $\Delta(B)= (1+o(1))np$
\item For all vertices $v \in B$ except for the two vertices with the largest degrees we have 
$$\Delta_2(B)-\deg(v) \geq \frac{\sqrt{np}}{\log n}.$$
\end{enumerate}
\end{lemma}
\begin{proof}    
(B2) is the bipartite version of Theorem 3.5 in \cite{frieze2015introduction}, and follows from the same proof.\\
(B3) is a weaker bipartite version of Theorem 3.15 in \cite{bollobas1998random}  proved by Bollob{\'a}s (it can be seen by taking $m$ to be a slowly increasing function). The same proof works for the bipartite case.\\
It is thus remains to prove (B1). We prove each case in Definition \ref{def:good} separately: 
   
    \begin{enumerate}[(a)]
            \item $s,t \geq n/\log^{2/3} n$. It follows from Chernoff's bounds and the union bound that the probability that there exist $S$ and $T$ of sizes $s$ and $t$, respectively, with $e(S,T)\notin \left(1\pm \delta \right)stp$ with $\delta=\frac{1}{\log^2 n}$
            is at most 
            \begin{align*}
            \sum_{s,t\geq n/\log^{2/3} n}\binom{n}{s}\binom{n}{t}2\exp\left(-\delta^2stp/3\right)&\leq 2 \cdot 4^n \exp\left(-\Theta(\delta^2n^{2}p/\log^{4/3}n)\right) \\
            &\leq 2 \cdot 4^n\exp\left(-\Theta(\delta^2 n\log^{5}n)\right) \\
            &= o(1). 
            \end{align*} 
            where the second from the last inequality holds since $np/\log^{4/3}n\geq \log^5n$.
            
            \item 
            $s \leq \frac{n}{\log^{2/3}n}$ and 
$s\sqrt{\log n} + t + 2 \geq n$. It follows from Chernoff's bounds (\cref{chernoff}) and the union bound that the probability there exist $S$ and $T$ with sizes $s$ and $t$, respectively, such that the number of edges between them satisfies $e(S,T)\notin \left(1 \pm \delta \right)stp$, where $\delta=\frac{1}{\log^2 n}$,
            is at most 
            \begin{align*}
            &\sum_{s=1}^{n/\log^{2/3}n}\sum_{t=n-2-s\sqrt{\log n}}^n \binom{n}{s}\binom{n}{t}2\exp\left( -\delta^2stp/3\right) \\
            &\leq 2n\sum_{s=1}^{n/\log^{2/3} n}{n \choose s}\binom{n}{s\sqrt{\log n}+2}\exp\left(-\delta^2 s(n-s\sqrt{\log n}-2)/3\right)\\ 
               &\leq 2n\sum_{s=1}^{n/\log^{2/3}n} n^{s+s\sqrt{\log n}+2}\exp\left(-\delta^2 sn/4\right) \\
               &= o (1) 
            \end{align*}

            \item Similar to the proof of (b).
        \end{enumerate}
        This completes the proof. 
\end{proof}

We will also need the following properties of $D_{n,p}$:
\begin{lemma}\label{typicaldirected} 
For $p \geq \frac{\log^{10} n}{n}$, the digraph $D:=D_{n,p}$ whp satisfies the following properties:
\begin{enumerate}[$(D1)$]
    \item Every vertex $v\in V(D)$ satisfies
    \begin{align*}
        (1-o(1))np\leq deg_D^+(v),\deg_D^{-}(v)\leq (1+o(1))np.
    \end{align*}
    \item Every set $X \subseteq V(D)$ of size $|X| \leq \frac{2\log^2 n}{p}$ satisfies
    $$e_D(X) \leq 12|X|\log^2n.$$
   
    \item Every pair of sets $X,Y \subseteq V(D)$ of sizes $|X|, |Y| \geq \frac{\log^{1.05} n}{p}$ satisfy
    $$e_D(X,Y) \leq (1+o(1))|X||Y|p.$$

\end{enumerate}
\end{lemma}
\begin{proof}
$(D1)$ is the directed version of Theorem 3.5 in \cite{frieze2015introduction}, and follows from the same proof. 
First, we prove $(D2)$. Observe that for any set $X$, since $e_D(X)\sim \mathrm{Bin}(|X|^2,p)$, the probability that $e_D(X) \geq 12|X|\log^2n$ is upper bounded by 
    \begin{align*}{ |X|^2 \choose 12|X|\log^2n}p^{12|X|\log^2n}&\leq \left(\frac{e|X|p}{12\log^2n}\right)^{12|X|\log^2n}\\ 
    & < 2^{-12|X|\log^2n}.
    \end{align*}
    Now, since there are at most 
    $${n \choose |X|}\leq n^{|X|}=2^{|X|\log n}$$ 
    sets of size $|X|$, it follows from the union bound that the probability to have such a set $X$ is at most 
    $$\sum_{|X|\leq \frac{2\log^2n}{p}}2^{|X|\log n-12|X|\log^2n}=o(1)$$
as desired. 
    
Now we prove $(D3)$. Given any two sets $X, Y,$ we have that
    $$\Exp\left( e_D(X,Y)\right)= |X||Y|p.$$
    By Chernoff's bound (\Cref{chernoff}), for any $\delta< 1$, 
    $$\Prob\left( e_D(X,Y)\notin (1\pm \delta)|X||Y| \right)< 2e^{-\delta^22|X||Y|p/3}=e^{-\Theta(|X||Y|p)}.$$
    Therefore, by applying the union bound we obtain that the probability to have such sets is at most
    \begin{align*}
        &\sum_{|X|,|Y|\geq \frac{\log^{1.05}n}{p}}\binom{n}{|X|}\binom{n}{|Y|}e^{-\Theta(|X||Y|p)}\leq\\ &\sum_{|X|,|Y|\geq \frac{\log^{1.05}n}{p}}\exp\left((|X|+|Y|)\log n-\Theta(|X||Y|p))\right)=o(1),
    \end{align*}
where the last inequality holds since for $|X|\geq |Y|\geq \frac{\log^{1.05}n}{p}$ we have
$$|X||Y|p\geq |X|\log^{1.05}n=\omega((|X|+|Y|)\log n).$$
This completes the proof.
\end{proof}

\section{Covering $B_{n,p}$ with perfect matchings}\label{sec:CoveringBnp}


In this section, we prove the following key lemma:
 \begin{lemma}\label{cor:covering} 

    Let $B:= B_{n,p}$ with $1/2 \geq p\geq \frac{\log^{10}n}{n}$. Then, whp, for all subgraphs $B'$ of $B$ obtained by deleting $\Delta(B)-\Delta_2(B)$ edges from the vertex of largest degree, there exist two families of perfect matchings of $B'$, $\mathcal M_1$ and $\mathcal M_2$, for which the following holds:
    \begin{enumerate}
    \item For $i\in \{1,2\}$ and for every $M\neq M'\in \mathcal M_i$, we have $M\cap M'=\emptyset$, and
    \item $\bigcup_{M\in \mathcal M_1\cup \mathcal M_2} M=E(B')$, and 
    \item $|\mathcal M_1|+|\mathcal M_2|=\Delta(B')$.
    \end{enumerate}
    \end{lemma}

Before proving Lemma \ref{cor:covering} we will need the following consequence of \Cref{orestheorem} when the bipartite graph $B$ is $\left(\frac{1}{\log n},p\right)$-good. 

\begin{lemma}\label{lemma:f-factor}
    Let $1/2 \geq p\geq \frac{\log^{10} n}{n}$ and $B$ be a $\left(\frac{1}{\log n},p\right)$-good bipartite graph. Let $M \leq \frac{np}{2\log^{4/3} n}$ be an integer and $f:V(B) \to \mathbb{N}$. Suppose that:
    \begin{enumerate}
        \item[(i)] $ f(v) \leq M$ holds for all $v\in V$. 
        \item[(ii)]  $f(v) \geq M/\sqrt{\log n}$
    \end{enumerate}
    Then, the graph $B$ contains an $f$-factor.
\end{lemma}
\begin{proof}
    
 By \Cref{orestheorem}, the graph $B$ contains an $f$-factor if and only if equation (\ref{eq:ore}) is satisfied. Let $S \subseteq X$ and $T \subseteq Y$. 
Lets first assume that $|S|,|T|\geq \frac{n}{\log^{2/3} n}$. Observe that because $B$ is $\left(\frac{1}{\log^2n},p\right)$-good, it holds that $e(S,T) \in (1\pm o(1))\frac{n^2p}{\log^{4/3} n}$. Further, observe that because $|S|$ is bounded above by $n$, and $f(x) \leq M$ for all $x$, we have that $n M \geq \sum_{x \in S}f(x)$. Now, we can see that equation (\ref{eq:ore}) holds, since
\begin{align*}
    e(S,T) + \sum_{y \in  Y \setminus T}f(y) \geq e(S,T) \geq (1+o(1))\frac{n^2p}{\log^{4/3} n} \geq n M \geq \sum_{x \in S}f(x).
\end{align*}

Therefore, we may assume that either $|S|< \frac{n}{\log^{2/3}n}$ or $|T|< \frac{n}{\log^{2/3}n}.$

Without loss of generality we assume that $|S|< \frac{n}{\log^{2/3}n}$. We can further assume that $\sum_{y \in Y \setminus T} f(y) < \sum_{x \in S}f(x)$ (otherwise equation (\ref{eq:ore}) immediately holds). Since $f(v) \geq \frac{M}{\sqrt{\log n}}$ for all vertices $v$,
\begin{align*}
    \frac{M}{\sqrt{\log n}}|Y\setminus T|\leq \sum_{y\in Y\setminus T}f(y)<\sum_{x \in S}f(x) \leq M|S|.
\end{align*}
Thus, we obtain that $|Y\setminus T|<|S|\sqrt{\log n}$ and consequently $|S|\sqrt{\log n}+|T|\geq n$. Hence, by part (b) of \Cref{def:good} we have that
\begin{align*}
    e(S,T) + \sum_{y \in  Y \setminus T}f(y) \geq e(S,T) \geq (1+o(1))p|S||T|> \frac{2np}{3}|S| \geq M|S|\geq \sum_{x \in S} f(x)
\end{align*}
Thus, equation (\ref{eq:ore}) holds.
\end{proof}

Also before proving \cref{cor:covering}, we will need the following result about finding regular subgraphs in $B_{n,p}$:
\begin{lemma}{\label{thm:matchingcovering1}}
    Let $1/2 \geq p\geq\frac{\log^{10} n}{n}$ and $B := B_{n,p}$ be a random bipartite graph. The following holds with high probability: For all subsets of edges $E\subseteq B$  with $|E|\leq \sqrt{np}\log n$, the subgraph $B\setminus E$ contains a $\left(1-\frac{2}{\log^2 n}\right)np$-regular subgraph. 
\end{lemma}
\begin{proof}
Let $B:=B_{n,p}$ with parts $X$ and $Y$, and let $r:=\left(1-\frac{2}{\log^2 n}\right)np$. In order to prove the theorem we wish to apply \Cref{lem:gale-ryser}. To this end, it is enough to show that for every non-empty sets $S\subseteq X$ and $T \subseteq Y$, 
\begin{align}\label{eq:galeryser}
    \max\{e(S,T)-|E|,0\}>r(|S|+|T|-n).
\end{align}
If $|S|+|T|< n$, then $r(|S|+|T|-n)< 0$ and consequently inequality (\ref{eq:galeryser}) holds. Thus, we may assume that $|S|+|T| \geq n$. 
Now, since $|S|+|T|\geq n$ and \cref{typicalbipartite} tells us that B is $\left(\frac{1}{\log^2 n}, p\right)$-good, 
$$e(S,T) \in \left(1\pm \frac{1}{\log^2 n}\right)|S||T|p.$$ 

Next, observe that  
$$|S||T|\geq  n(|S|+|T|-n).$$ 

Indeed, the above is equivalent to 
$$n^2\geq |S|(n-|T|)+n|T|$$
which is true because $|S|\leq n$.

Moreover, since $S$ and $T$ are non empty, and since at least one of them is large, we have $|E|=o(\frac{p|S||T|}{\log^2 n})$. All in all, we have
\begin{align*}
    e(S,T)-|E|&\geq \left(1-\frac{1}{\log^2 n}\right)|S||T|p-|E|\\
    &> \left(1-\frac{2}{\log^2 n}\right)np(|S|+|T|-n)\\
    &=r(|S|+|T|-n)
\end{align*}
which proves inequality (\ref{eq:galeryser}). This completes the proof.
\end{proof}

 \begin{proof} [Proof of Lemma \ref{cor:covering}.]
Let $B:=B_{n,p}$, and let $B'$ be a subgraph of $B$ obtained by deleting $\Delta(B)-\Delta_2(B)$ edges from the vertex of largest degree. First, we find a family of perfect matchings that ``almost'' cover $B'$. This is done by applying \Cref{thm:matchingcovering1} to obtain $W'$, an $np\left(1-\frac{2}{\log^2 n}\right)$-regular subgraph of $B'$, and then applying \Cref{thm:hallscor} to decompose $W'$ into perfect matchings. We call this set of matchings $\cM_1$. 
Next, let $W$ be the subgraph of $B'$ whose edges are not covered by any matching in $\mathcal M_1$. 
 In order to complete the proof, we need to find a family of perfect matchings, $\mathcal M_2$, of size $\Delta(B')-|\mathcal M_1|$, that covers $W$. 
 We will find $\mathcal M_2$ by finding an $r:=\Delta(B')-|\mathcal M_1|$ regular subgraph of $B'$ that contains all the edges of $W$. We will decompose this into perfect matchings using \cref{thm:hallscor}. To this end, we first show that $W'$, the $np\left(1-\frac{2}{\log^2 n}\right)$-regular graph obtained in \Cref{thm:matchingcovering1}, is $\left(\frac{1}{\log n},p\right)$-good:
 

\begin{claim}\label{lem:stillnice}
    With high probability, for all choices of $B'$, the $np\left(1-\frac{2}{\log^2 n}\right)$-regular graph obtained in \Cref{thm:matchingcovering1} is $\left(\frac{1}{\log n},p\right)$-good. 
\end{claim}

\begin{proof} 

\cref{typicalbipartite} implies that $B$ is $\left(\frac{1}{\log^2 n},p\right)$-good. It is sufficient to show that when $S$ and $T$ satisfy the size requirements of (a), (b), or (c) in \Cref{def:good}, at most $\frac{stp}{2\log n}$ edges between $S$ and $T$ are in $B$ but not in $W'$. If any of the size conditions of (a), (b) or (c) hold, then either $s \geq \frac{n}{\log^{2/3} n}$ or $t \geq \frac{n}{\log^{2/3} n}$. Without loss of generality, suppose $s \geq \frac{n}{\log^{2/3}n}$.
Since $\Delta(B)\leq np+\frac{np}{\log^2n}$, and since $W'$ is $(1-\frac{2}{\log^2n})np$-regular, it follows that at most $\frac{3np}{\log^2 n}$ edges were deleted from each vertex in $T$. In total, the number of deleted edges between $S$ and $T$ is at most $3t\frac{np}{\log^2 n}<\frac{stp}{2\log n}$.
\end{proof}

        
       We are now ready to complete the proof of \cref{cor:covering} by finding an $r$-regular subgraph of $B'$ that covers $W$, the graph of edges not covered by $\mathcal M_1$. By \cref{lem:stillnice}, $W'$ is $\left(\frac{1}{\log n},p\right)$-good, so we will apply \Cref{lemma:f-factor} to we find an $f$-factor of $W'$, where $f(v)=\Delta(W)-\deg_{W}(v)$. Note that $\max_v\{f(v)\}<\frac{3np}{\log^2{n}}< \frac{np}{2\log^{4/3}n}$, and $\min_v\{f(v)\}> \frac{np}{2\log^2{n}}$, so we may apply \Cref{lemma:f-factor} with $M=\frac{3np}{\log^2{n}}$. Call this $f$-factor $H$.
       The critical point is that the graph with the edge set $W \cup E(H)$ is bipartite and $\Delta(W)$-regular, and therefore the edges can be partitioned into $\cM_2$, a set of $\Delta(W)$ perfect matchings covering $W$.
    \end{proof}

\section{Finding an almost covering of $D_{n,p}$ by linear forests}\label{sec:coverforest}

In this section our goal is to prove that the random digraph $D_{n,p}$ can be typically decomposed into a family of \emph{linear forests} and a ``sparse'' graph. A digraph $F$ is said to be a \emph{linear forest} if $F$ can be partitioned into disjoint directed paths $F=\bigcup_{i=1}^\ell P_i$, where we also consider isolated vertices as paths. We say that $F$ has $\ell$ connected components if $F$ is the union of $\ell$ disjoint directed paths. For a digraph $D$ on $n$ vertices, let $\Delta_1(D)\geq \Delta_2(D)\geq \ldots \geq \Delta_{2n}(D)$ be its ordered degree sequence of in-degrees and out-degrees. 
For simplicity, from now on, we will assume that the maximum degree vertex of the directed graph $D_{n,p}$ is an out-degree. The main theorem will still hold if the maximum degree vertex is an in-degree, since by flipping the orientation of each edge, we obtain a graph whose maximum degree vertex is an out-degree, and any Hamilton cycle on the ``flipped" graph is a Hamilton cycle on the original graph. The main result of this section can be precisely stated as follows.
\begin{lemma}\label{lem:almostforest}
Let $D:=D_{n,p}$ with $1/2 \geq p\geq \frac{\log^{10} n}{n}$. Then, with high probability the following holds: Let $x_{\star} \in V(D)$ be a vertex of largest degree (so in particular we have $\Delta_1(D)=\deg_D^+(x_\star)$).
Then $D$ can be covered by a family $\cF=\{F_1,\ldots,F_t\}$ of linear forests for $\Delta_2(D)-3\leq t\leq \Delta_2(D)$, a digraph $R$ and a set of edges $E=D\setminus \left(R\cup\left(\bigcup_{i=1}^t F_t\right)\right)$ satisfying
\begin{enumerate}
    \item[(i)] For every $1\leq i \leq t$, there exists vertex $y_i\in V(D)$ such that $\orarrow{x_\star y_i}\in F_i$ 
    \item[(ii)] For every $1\leq i \leq t$, the forest $F_i$ has at most $4\log n$ components.
    \item[(iii)] The edge set $E$ consists of $\Delta_1(D)-t$ out-edges from $x_{\star}$ 
    \item[(iv)] The digraph $R$ satisfies $\Delta(R)= O\left((np\log^4 n)^{1/3}\right)$.
\end{enumerate}
\end{lemma}

To prove \Cref{lem:almostforest} we are going to consider a different random model that has the same distribution as $D_{n,p}$. In this model we construct a random digraph using a random permutation and a random bipartite graph as described below. Recall that in a bipartite graph, the edge $ij$ refers to $i$ in part $X$ and $j$ in part $Y$, so $ij \neq ji$. 

\begin{definition}\label{def:model}
Given a balanced bipartite graph $B=(X\cup Y,E)$ with $|X|=n$, and a permutation $\pi \in S_n$, we define the digraph $D_{\pi}(B)$ as the digraph with vertex set $[n]$ and edge set given by
\begin{align*}
    E(D_{\pi}(B))=\{\orarrow{ij}\in [n]^2:\: i\pi^{-1}(j)\in B \text{ and } i\neq j \}.
\end{align*}
\end{definition}

In other words, the digraph $D_{\pi}(B)$ is the loopless digraph obtained by shuffling the vertices of the part $Y$ by the permutation $\pi$ and associating every edge of the new relabeled bipartite graph with a directed edge. The first observation is that the random model for directed graph $D_{\pi}(B_{n,p})$ obtained by choosing a random bipartite graph in $B_{n,p}$ and then independently choosing permutation uniformly at random in $S_n$ has same distribution as $D_{n,p}$.

\begin{fact}\label{fact:model}
Let $\pi \in S_n$ be a permutation chosen uniformly at random. Then $D_{\pi}(B_{n,p})$ has the same probability distribution as $D_{n,p}$.
\end{fact}

The main observation is that for a matching $M$, the digraph $D_{\pi}(M)$ is a $1$-factor, i.e. a disjoint union of directed cycles and isolated vertices. By removing one edge from each directed cycle, one can obtain a linear forest. This makes a natural connection between families of matchings in $B$ and families of linear forests in $D_{\pi}(B)$. Before we prove \Cref{lem:almostforest} we will state some facts about cycles in random permutations. Given a permutation $\pi \in S_n$, let $C(\pi)$ be the number of cycles in $\pi$. The next result allows us to show that typically a random permutation does not have too many cycles.

\begin{theorem}[\cite{ford2022cycle}, Theorem 1.24]\label{thm:expectedcycles}
Let $\pi\in S_n$ be a permutation chosen uniformly at random, then $\EE(2^{C(\pi)})=n+1$.
\end{theorem}

Since each directed cycle in $D_{\pi}(M)$ corresponds to a cycle in the permutation $\pi$, we can use the last theorem to bound the number of connected components in $D_{\pi}(M)$.

\begin{corollary}\label{cor:comp}
Let $M$ be a perfect matching and $\pi \in S_n$ be a permutation chosen uniformly at random. Then with probabiliy $1-1/n^3$, the $1$-factor $D_{\pi}(M)$ has at most $4\log n$ components.
\end{corollary}

\begin{proof}
Let $m \in S_n$ be the permutation where $m(i)=j$ when $ij \in M$. The number of components in $D_\pi(M)$ is equal to the number of cycles in the permutation $\pi(m)$. Since $\pi$ is chosen uniformly at random, regardless of $m$, $\pi(m)$ is distributed uniformly. We may now apply \cref{thm:expectedcycles} to obtain $\EE(2^{C(\pi(m))})=n+1$. By Markov's inequality,

$$\Prob\left( 2^{C\left(\pi(m)\right)} \geq (n+1)^4 \right) \leq \frac{1}{(n+1)^3},$$
implying:
$$\Prob\left( C\left(\pi(m)\right) \geq 4\log(n+1) \right) \leq \frac{1}{(n+1)^3}.$$
\end{proof}

As mentioned before, our strategy to obtain linear forests will be to delete some edges of the $1$-factors obtained in $D_{\pi}(B)$. The edges are going to be deleted at random satisfying certain properties. Using the following result from \cite{fhm} we can bound the degrees of the graph obtained by removing these edges.

\begin{lemma}\label{lem:probcycles}
Let $\cM=\{M_1,\ldots,M_r\}$ be a family of $r$ disjoint matchings. For each $w \in [n]$ and $1\leq i \leq r$, let $c_{i,w}$ be the length of the cycle containing $w$ in $D_{\pi}(M_i)$. If $w$ is isolated in $D_{\pi(M_i)}$, define $c_{w,i}=1$. Then whp, for every vertex $w \in [n]$ it holds that
\begin{align*}
    \sum_{i=1}^r\frac{1}{c_{i,w}}=O\left((r\log^4 n)^{1/3}\right)
\end{align*}
\end{lemma}

The proof of \Cref{lem:probcycles} is based on the following technical claim that we postpone to the appendix (the proof is taken from \cite{fhm} but since this paper has not been submitted yet, we include its proof for the convenience of the reader).

\begin{proposition}\label{prop:ugly}
For every $w \in [n]$, it holds that
\begin{align*}
    \EE\left(\left(\sum_{i=1}^r\frac{1}{c_{i,w}}\right)^3\right)=O\left(\frac{r}{n}\log^3 n\right).
\end{align*}
\end{proposition}

The proof of the lemma now follows from a third moment argument.

\begin{proof}[Proof of \Cref{lem:probcycles}]
Given any $w\in [n]$, by Markov's inequality and \cref{prop:ugly}, we have:
$$\Prob\left(\left(\sum_{i=1}^r\frac{1}{c_{i,w}}\right)^3 \geq O\left(r\log^4 n\right)\right) \leq \frac{1}{n\log n}.$$
Therefore, $$\Prob\left(\sum_{i=1}^r\frac{1}{c_{i,w}} \geq O\left((r\log^4 n)^{1/3}\right)\right) \leq \frac{1}{n\log n}.$$
The result now follows by a union bound.
\end{proof}

Before proving Lemma \ref{lem:almostforest}, we need the following observation that in a typical $D:=D_{n,p}$ there is a gap between $\Delta_1(D)$ and $\Delta_2(D)$.

\begin{proposition}\label{prop:dgap}
Let $D:=D_{n,p}$ with $1/2 \geq p\geq \frac{\log^{10} n}{n}$. Then whp it holds that
\begin{align*}
    \Delta_1(D)-\Delta_2(D)\geq \frac{\sqrt{np}}{2\log n}.
\end{align*}
\end{proposition}

\begin{proof}
By Fact \ref{fact:model} it suffices to show the proposition holds whp for $D_{\pi}(B_{n,p})$. We prove this by first exposing $B:=B_{n,p}$ and then later exposing $\pi \in S_n$. Let $B:=B_{n,p}$, so with high probability, it has the properties of \Cref{typicalbipartite}. Let $x_1$ and $x_2$ be the vertices of largest and second largest degree in $B$, respectively. Assume without loss of generality that both $x_1$, $x_2 \in X$, i.e., both vertices are in the first partition of $B$. By property (B3) of \Cref{typicalbipartite} we have that
\begin{align*}
    \deg_B(x_1)-\deg_B(x_2)\geq \frac{\sqrt{np}}{\log n}.
\end{align*}
Now reveal $\pi \in S_n$ uniformly at random and consider $D:=D_{\pi}(B)$. By the construction of Definition \ref{def:model}, we have that $\deg_B(x_1)-1\leq \deg^+_D(x_1) \leq \deg_B(x_1)$ and $\deg_B(x_2)-1\leq \deg^+_D(x_2) \leq \deg_B(x_2)$. Thus,
\begin{align*}
    \Delta_1(D)-\Delta_2(D)=\deg_D^+(x_1)-\deg_D^+(x_2)\geq \frac{\sqrt{np}}{\log n}-1\geq \frac{\sqrt{np}}{2\log n}.
\end{align*}
Since the last bound holds for any $\pi\in S_n$, the result now follows. 
\end{proof}

We are now ready to prove \Cref{lem:almostforest}.

\begin{proof}[Proof of \Cref{lem:almostforest}]
By Fact \ref{fact:model} it suffices to show that the result holds whp for $D_{\pi}(B_{n,p})$. As before, we will first expose $B:=B_{n,p}$ and then later $\pi \in S_n$. Let $B:=B_{n,p}$, and observe that $B$ simultaneously satisfies \Cref{typicalbipartite} and \Cref{cor:covering} with high probability. Let $x_\star$ be the vertex of largest degree in $B$ and assume without loss of generality that $x_\star \in X$. By property (B3) of \Cref{typicalbipartite}, we have that
\begin{align*}
    d:=\Delta(B)-\Delta_2(B)\geq \frac{\sqrt{np}}{\log n}.
\end{align*}
Let $E'\subseteq B$ be an arbitrary subset of $d$ edges in $B$ containing $x_\star$. Set $B'=B\setminus E'$. Then by \Cref{cor:covering}, there exists two families of perfect matchings $\cM_1$ and $\cM_2$ satisfying
\begin{enumerate}
    \item[(a)] For $j\in \{1,2\}$ and for every $M\neq M'\in \mathcal M_j$, we have $M\cap M'=\emptyset$,
    \item[(b)] $B'=\bigcup_{M\in \mathcal M_1\cup \mathcal M_2} M$, 
    \item[(c)] $\Delta(B')=\Delta_2(B)=|\mathcal M_1|+|\mathcal M_2|$.
\end{enumerate}

We now reveal $\pi \in S_n$. Since $|\cM_1|+|\cM_2|=\Delta(B')\leq n$, by a simple union bound we obtain that whp $\pi \in S_n$ simultaneously satisfies \Cref{cor:comp} for every $M\in \cM_1\cup \cM_2$ and \Cref{lem:probcycles} for both $\cM_1$ and $\cM_2$. Take such an instance of $\pi \in S_n$. Let $D:=D_{\pi}(B)$, and observe by \cref{fact:model} that $D$ is distributed identically to $D_{n,p}$. 

For each $j\in \{1,2\}$, let $r_j=|\cM_j|$ and $\cM_j=\{M_1,\ldots, M_{r_j}\}$ be the family of perfect matchings. Consider the $1$-factors $D_{\pi}(M_1), \ldots, D_{\pi}(M_{r_j})$. For each $j$, since $\cM_j$ is a disjoint family of matchings and $\deg_{D_\pi(B)'}^+(x_\star) \in \{\delta(B'),\delta(B')-1\}$, there exists at most one $D_{\pi}(M_j)$ such that $x_\star$ is an isolated vertex. If such a matching exists, remove it from the family. After reordering we obtain a family $\cM'_j=\{M_1,\ldots,M_{t_j}\}$ where $r_j-1\leq t_j\leq r_j$ and for $1\leq i \leq t_j$, there exists $y_i \in V(D)$ with $\orarrow{x_{\star}y_i} \in D_{\pi}(M_i)$. 

We will now construct a graph $R_j$ and a family of linear forest $\cF_j=\{F_1,\ldots,F_{t_j}\}$ satisfying the following properties:
\begin{enumerate}
    \item[(F1)] For $1\leq i \leq t_j$, the edge $\orarrow{x_\star y_i}\in F_i$.
    \item[(F2)] For $1\leq i \leq t_j$, the forest $F_i$ has at most $4\log n$ components.
    \item[(F3)] $D_{\pi}(\cM_j)=R_j\cup \left(\bigcup_{i=1}^{t_j} F_i\right)$
    \item[(F4)] $\Delta(R_j)\leq O\left((t_j\log^4 n)^{1/3}\right)$.
\end{enumerate}
For each $1\leq i\leq t_j$ and cycle $C\subseteq D_{\pi}(M_i)$, select an edge from $C\setminus\{\orarrow{x_\star y_j}\}$ independently and uniformly at random and add it to $R_j$. Set $F_i= D_{\pi}(M_i)\setminus R_j$. Clearly, $F_i$ is a linear forest and $\orarrow{x_{\star}y_j} \in F_i$ for $1\leq i \leq t_j$. This satisfies property (F1). Property (F2) follows from \Cref{cor:comp}, since the number of components in $D_{\pi}(M_i)$ and $F_i$ is the same. Property (F3) is immediate from the construction.

Finally, we will show that property (F4) holds. Let $w\in V(D)$ and let $c_{i,w}$ be the length of the cycle containing $w \in D_{\pi}(M_i)$ (If $w$ is an isolated vertex, then set $c_{i,w}=1$). Then the degree of $w$ in $R_j$ can be written as $\deg^{+}_{R_j}(w)=\sum_{i=1}^{t_j}X_i$, where $X_i$ are independent $0/1$-random variables given by $X_i=\deg^{+}_{R_j\cap D_{\pi}(M_i)}(w)$. That is, $X_i$ counts if the out-edge from $w$ in the $1$-factor $D_{\pi}(M_i)$ was selected to $R_j$. One can observe that
\begin{align*}
    \EE(X_i)=\begin{cases}
        0, &\quad \text{if $w=x_\star$ or $w$ is an isolated vertex in $D_{\pi}(M_i)$}\\
        \frac{1}{c_{i,w}-1}, &\quad \text{if $w$ and $x_\star$ belongs to the same cycle in $D_{\pi}(M_i)$}\\
        \frac{1}{c_{i,w}}, &\quad \text{otherwise},
    \end{cases}
\end{align*}
and consequently $\EE(X_i)\leq 2/c_{i,w}$ for $1\leq i\leq t_j$ and $w\in V(D)$. By \Cref{lem:probcycles}, the expected value of $\deg^+_{R_j}(w)$ can be bounded by
\begin{align*}
    \EE(\deg^+_{R_j}(w))\leq \sum_{i=1}^{t_j} \frac{2}{c_{i,w}} \leq O\left((t_j\log^4 n)^{1/3}\right).
\end{align*}
Thus, by \Cref{chernoff2}, we have that
\begin{align*}
    \PP\left(\deg^+_{R_j}(w)>O\left((t_j\log^4 n)^{1/3}\right)\right)\leq \exp\left(-O\left((t_j\log^4 n)^{1/3}\right)\right)=o\left(\frac{1}{n}\right).
\end{align*}
Hence, by a union bound we obtain that whp $\Delta^+(R_j)\leq O\left((t_j\log^4 n)^{1/3}\right)$. Similarly, whp we have that $\Delta^-(R_j)\leq O\left((t_j\log^4 n)^{1/3}\right)$. Therefore, there exists a choice of $R_j$ such that (F4) holds.

To finish the proof, take $R:=R_1\cup R_2 \cup D_{\pi}\left(\left(\cM_1\setminus \cM'_1\right)\cup \left(\cM_2\setminus \cM'_2\right)\right)$, $E=D_{\pi}(E')$ and $\cF:=\cF_1\cup \cF_2$. Note that $\Delta_2(D)-1\leq \Delta_2(B)\leq \Delta_2(D)$, since there might be an edge in $B$ between $x_\star$ and $\pi(x_\star)$, and $D$ has no self loops.  
Therefore, since $\Delta_2(B)=r_1+r_2$ and $r_j-1\leq t_j\leq r_j$ for $1\leq j\leq 2$, we have that
\begin{align*}
    \Delta_2(D)-3\leq \Delta_2(B)-2\leq t=|\cF|=t_1+t_2= \Delta_2(B)\leq \Delta_2(D).
\end{align*}
Property (i) and (ii) follows immediately from (F1) and (F2), respectively. By construction, $E$ consists of out-edges from $x_\star$ and (F3) gives us
\begin{align*}
    E=D_{\pi}(E')=D_{\pi}(B\setminus B')=D_{\pi}(B)\setminus D_{\pi}(B')=D\setminus \left(R\cup\left(\bigcup_{i=1}^t F_t\right)\right).
\end{align*}
Hence, by the fact that $\deg^+_R(x_\star)=0$, we have property (iii), i.e., $|E|=\Delta_1(D)-t$. Property (iv) follows from (F4), since
\begin{align*}
    \Delta(R)\leq O\left((t_1\log^4 n)^{1/3}+(t_2\log^4 n)^{1/3}\right)+2\leq O\left((t\log^4 n)^{1/3}\right).
\end{align*}
This concludes the proof. 
\end{proof}

\section{Covering linear forests by Hamilton cycles}\label{sec:coverhamilton}

In this section we prove some lemmas that will enable us to cover linear forests by Hamilton cycles. Our main technical tool in this section is the main result of \cite{ferber2017robust} that asserts that appropriate pseudorandom digraphs are Hamiltonian. We start with a definition.

\begin{definition}\label{def:pseudorandom}
     A directed graph $H$ on $n$ vertices is \emph{$(n,\alpha,p)$-pseudorandom} if the following properties hold:
\begin{enumerate}
    \item[(P1)] For every $v \in V(H)$, we have 
    \begin{align*}
        \min\{\deg^+_H(v),\deg^-_H(v)\}\geq (1/2+2\alpha)np
    \end{align*}
    \item[(P2)] For every $X \subseteq V(H)$ of size $|X| \leq \frac{\log^2n}{p}$, we have
    \begin{align*}
        e_H(X) \leq |X| \log^{2.1}n
    \end{align*}
    \item[(P3)] For every two disjoint $X, Y \subseteq V(H)$ with $|X|,|Y| \geq \frac{\log^{1.1}n}{p}$, we have 
    \begin{align*}
        e_H(X,Y)\leq (1+\alpha/2)p|X||Y|
    \end{align*}
\end{enumerate}
\end{definition}

In other words, an $(n,\alpha,p)$-pseudorandom digraph is a graph with a certain minimum degree and ``well-distributed'' edges. In \cite{ferber2017robust}, the authors proved that these properties are sufficient for containing a directed Hamilton cycle.

\begin{theorem}[\cite{ferber2017robust}, Theorem 3.2]\label{thm:pseudorandom}
    Let $\alpha>0$ and $1/2 \geq p\geq \frac{\log^9{n}}{n}$. Then for sufficiently large $n$, every $(n,\alpha,p)$-pseudorandom directed graph is Hamiltonian.
\end{theorem}

Now, in order to cover a directed linear forest by a Hamilton cycle, we construct an auxiliary digraph such that each Hamilton cycle in this digraph corresponds to a cover of the linear forest by a Hamilton cycle. Then, we show that this auxiliary digraph satisfies the notion of pseudorandomness in \Cref{def:pseudorandom} and therefore by Theorem \ref{thm:pseudorandom} it contains a Hamilton cycle. Before the formal description we need to introduce some notation. 

A directed linear forest $F$ can be written as an union of disjoint directed paths $F=\bigcup_{i=1}^\ell P_i$. Given a path $P=(x_1,\ldots,x_m)$, we define the \emph{source} $s(P)$ of the path $P$ as the first vertex in the path, i.e., $s(P)=x_1$.  Similarly, the \emph{sink} $t(P)$ of $P$ is defined as the last vertex of the path, i.e., $t(P)=x_m$. Recall that we consider isolated vertices to be paths of length $0$. We now are able to define our auxiliary ``contracting'' digraph.

\begin{definition}\label{def:auxgraph}
Give a directed graph $D$ and a linear forest $F=\bigcup_{i=1}^\ell P_i$, we define the auxiliary directed graph $H_{F,D}$ as the graph with vertex set $V(H_{F,D})=\{P_j\}_{j\in [\ell]}$ and edge set given by
\begin{align*}
    H_{F,D}= \left\{\orarrow{P_iP_j} :\: \overrightarrow{t(P_i)s(P_j)} \in D \right\}
\end{align*}
That is, $H_{F,D}$ is the graph where the vertices are the paths and we add an edge from $P_i$ to $P_j$ if there is an edge in $D$ from the sink of $P_i$ to the source of $P_j$.
\end{definition} 

Note by the last definition that the number of vertices of $H_{D,F}$ is exactly the number of components of $F$. The next result shows that if our host digraph $D$ is typical and $F$ has a ``relatively large'' number of connected components, then to check that $F$ can be covered by a Hamilton cycle is the same as checking that $H_{F,D}$ satisfies property (P1) of \Cref{def:pseudorandom}.

\begin{proposition}\label{prop:auxgraph} 

Let $D:=D_{n,p}$ with $1/2 \geq p\geq \frac{\log^{10} n}{n}$. Then, whp the following holds: Suppose that $F$ is a directed linear forest with $\ell\geq \frac{n^{1/3}}{2p^{2/3}}$
components and that its auxiliary digraph $H_{F,D}$ satisfies $\delta(H_{F,D}) \geq 3p\ell/4$. Then there exists a Hamilton cycle in $D\cup F$ covering all the edges of $F$.
\end{proposition}

\begin{proof}
First, note that it suffices to show that $H_{F,D}$ is Hamiltonian. Indeed, suppose that $H_{F,D}$ has a Hamilton cycle $C$. Let $C'$ be the cycle in $D\cup F$ obtained by replacing the vertex $P_i$ in $H_{F,D}$ by the path $P_i$ in $F$. Clearly $C'$ is a directed cycle containing all the vertices of $V(D)$ (since $V(F)=\bigcup_{i=1}^\ell P_i$) and $F\subseteq C'$. Therefore $C'$ is a Hamilton cycle covering $F$. 

We claim that $H_{F,D}$ is $(\ell,1/8,p)$-pseudorandom and consequently by \Cref{thm:pseudorandom} that it is Hamiltonian. First, note that 
\begin{align*} 
\frac{\log^9\ell}{\ell}  \leq \frac{\log^9 \left(\frac{n^{1/3}}{2p^{2/3}}\right)}{\frac{n^{1/3}}{2p^{2/3}}} 
 \leq \frac{\log^9 n -2^9\log^9 \left(2^{3/2}p\right)}{3^9n^{1/3}}2p^{2/3}
 < p.
\end{align*}
The first inequality holds because $\frac{\log^9 x}{x}$ is decreasing, and the last inequality holds because $p> \frac{\log^{10} n}{n}$.

Now that we've verified $p>\frac{\log^9\ell}{\ell}$, we just need to check properties (P1), (P2) and (P3) of \Cref{def:pseudorandom}. Property (P1) follows immediately from the hypothesis of the proposition applied to $\alpha=1/8$. We now prove property (P2). Let $X\subseteq V(H_{D,F})$ of size $|X|\leq \frac{\log^2 \ell}{p}$. Consider the set $X'\subseteq V(D)$ defined by
\begin{align*}
    X' := \left\{x \in V(D) :\: x= s(P) \text{ or } x= t(P) \text{ for } P \in X \right\}
\end{align*}
That is, $X'$ is the set of all sources and sinks of paths in $X$. Note by definition, that $|X'|\leq 2|X| \leq \frac{2\log^2 \ell}{p}\leq \frac{2\log^2 n}{p}$, since $\ell\leq n$. Therefore, by property (D2) of \Cref{typicaldirected} we have that $e_D(X')\leq 12|X'|\log^2 n$. Since any edge of $H_{F,D}[X]$ corresponds to an edge in $D[X']$ we obtain that $e_{H_{F,D}}(X) \leq e_D(X')$. Therefore, for $\ell\geq \frac{n^{1/3}}{2p^{2/3}}$
, we have that
\begin{align*}
    e_{H_{F,D}}(X) \leq 12|X'|\log^2 n\leq 24|X|\log^2n \leq 216|X|\log^2 \ell \leq |X|\log^{2.1} \ell,
\end{align*}
and property (P2) holds.

To prove property (P3), let $X, Y$ be disjoint subsets of $V(H_{F,D})$ with $|X|,|Y| \geq \frac{\log^{1.1}{\ell}}{p}$. Let $X'$ and $Y'$ be subsets of $V(D)$ defined by
\begin{align*}
    X' = \left\{t(P) \in V(D) :\: P \in X \right\} \quad \text{and} \quad Y' = \left\{s(P) \in V(D) :\: P \in Y \right\}
\end{align*}
That is, $X'$ is the set of sinks of $X$ and $Y'$ is the set of sources of $Y$. Note that $|X'|=|X|$ and $|Y'|=|Y|$. Thus, we have $|X'|, |Y'|\geq\frac{ \log^{1.1}\ell}{p}\geq \frac{ \log^{1.05}n}{p}$. Hence, by property (D3) of \Cref{typicaldirected} 
\begin{align*}
e_{H_{F,D}}(X,Y)=e_D(X',Y')\leq \left(1+\frac{1}{16}\right)p|X||Y|,
\end{align*}
concluding the proof.
\end{proof}

The following result is the main technical lemma of the section. Roughly speaking, it states that given a large family of linear forests with small number of components, one can almost cover it with a family of Hamilton cycles in a way that guarantees the remaining edges form a sparse graph with small maximum degree. Moreover, for technical reasons, the choice of Hamilton cycles is done to contain specific edges of the linear forests. 

\begin{lemma}\label{lem:keycover}
Let $D:=D_{n,p}$ with $1/2 \geq p\geq\frac{\log^{10} n}{n}$, and let $np/2 \leq t \leq 3np/2$ be an integer. The following holds whp: Suppose that $\cF=\{F_1,\ldots,F_t\}$ is a family of directed linear forests and $x\in V(D)$ is a vertex satisfying
\begin{enumerate}
    \item[(i)] For each $1\leq i \leq t$, the linear forest $F_i$ has at most $4\log n$ components.
    \item[(ii)] For each $1\leq i \leq t$, there exists a vertex $y_i \in V$ such that $\orarrow{xy_i}\in F_i$.
\end{enumerate}
Then there exists a family of Hamilton cycles $\cC=\{C_1,\ldots,C_t\}$ such that $\orarrow{xy_i} \in C_i$ and $C_i$ is a cycle in $D\cup F_i$. Moreover, if $R=\bigcup_{i=1}^t\left(F_i\setminus C_i\right)$ is the digraph of non-covered edges of $\cF$ by $\cC$, then $\Delta(R)\leq 7t^{1/3}$.
\end{lemma}

\begin{proof}
Let $V:=V(D)$. Note that whp $D:=D_{n,p}$ simultaneously satisfies \Cref{typicaldirected} and \Cref{prop:auxgraph}. 
We first remove a small amount of edges of $\cF$ at random. Let $R$ be the graph obtained by selecting edges from $\bigcup_{i=1}^t F_i\setminus\{\orarrow{xy_i}\}$ independently with probability $q=t^{-2/3}$. Set $U_i=F_i\setminus R$. We now prove three properties (Claims \ref{clm:components}, \ref{cl:auxgraph}, and \ref{cl:smalldegree}) that are satisfied by $R$ and the family of linear forests $\{U_i\}_{i=1}^t$. Let $m_i$ and $\ell_i$ and be number of connected components of $F_i$ and $U_i$, respectively.

\begin{claim}\label{clm:components}
 For $1\leq i\leq t$ and $\epsilon>0$, whp  it holds that $(1-\epsilon)nq\leq \ell_i \leq (1+\epsilon)nq$.    
\end{claim}
\begin{proof}
Note that every edge removed from $F_i$ creates a new component. Therefore, the number of components in $U_i$ is exactly $m_i+|F_i \cap R|$. We claim that whp $(1-\epsilon/2)nq\leq |F_i\cap R|\leq (1+\epsilon/2)nq$ for $1\leq i \leq t$. Indeed, since $F_i$ has at most $4\log n$ components, we have $|F_i|= n-4\log n=(1-o(1))n$. Chernoff's bound (\cref{chernoff}) and the fact that $t\geq np/2$ gives us that
\begin{align*}
    \PP\left(\big||F_i\cap R|-q|F_i|\big|>\frac{\epsilon nq}{3}\right)\leq \exp\left(-\frac{\epsilon^2 nq}{27}\right)=\exp\left(-\frac{\epsilon^2n}{27t^{2/3}}\right)=o\left(\frac{1}{n}\right).
\end{align*}
Hence, by a union bound, we have that $(1-\epsilon/2)nq\leq |F_i\cap R|\leq (1+\epsilon/2)nq$ holds with probability $1-o(1)$. This implies the bound
\begin{align*}
    (1-\epsilon)nq\leq |F_i\cap R|\leq \ell_i=m_i+ |F_i\cap R|\leq (1+\epsilon/2)nq+4\log n\leq (1+\epsilon)nq
\end{align*}
by our choice of $q$.
\end{proof}

Next, we check condition (P1) of \Cref{def:pseudorandom} for the auxiliary graph $H_{U_i,D}$. For each $1\leq i \leq t$, we write $F_i=\bigcup_{j=1}^{m_i} P_{ij}$ and $U_i=\bigcup_{j=1}^{\ell_i} Q_{ij}$, where $\{P_{ij}\}_{j=1}^{m_i}$ and $\{Q_{ij}\}_{j=1}^{\ell_i}$ are the disjoint paths of $F_i$ and $U_i$, respectively (recall that some of those paths might be isolated vertices). 

\begin{claim}\label{cl:auxgraph}
For $1\leq i \leq t$, whp we have that $\delta(H_{U_i,D})\geq 3p\ell_i/4$.    
\end{claim}
\begin{proof}
We will prove that whp $\delta^+(H_{U_i,D})\geq 3p\ell_i/4$ for $1\leq i \leq t$. Fix an index $1\leq i \leq t$. Let $S(F_i)=\{s(P_{ij}):\: 1\leq j \leq \ell_i\}$ be the set of sources of the paths in $F_i$ and $S(Q_i)=\{s(Q_{ij}):\: 1\leq j \leq \ell_i\}$. For a vertex $v\in V$, we define
\begin{align*}
    g(v):=\left|\left(N_{D}^{+}(v)\setminus \left(S(F_i)\cup\{y_i\}\right)\right)\cap S(U_i)\right|,
\end{align*}
where $N_D^{+}(v)$ is the out-neighborhood of $v$ in the digraph $D$. Note that 
\begin{align*}
    \delta^+(H_{U_i,D})=\min_{t(Q_{ij}):\: 1\leq j\leq \ell_i }|N_D^+(t(Q_{ij}))\cap S(U_i)|\geq \min_{v\in V} g(v),
\end{align*}
and therefore we just need to estimate $g(v)$.

We claim that $g(v)\geq 3p\ell_i/4$ holds for every $v\in V$ with probability at least $1-o(1/n)$. Note that for every vertex $w \in V\setminus \left(S(F_i)\cup\{y_i\}\right)$, there exists a unique vertex $z_w$ such that $\orarrow{z_w w} \in F_i\setminus\{\orarrow{x y_i}\}$. Hence, the set $\left(V\setminus \left(S(F_i)\cup\{y_i\}\right)\right) \cap S(U_i)$ can be characterized as
\begin{align*}
\left(V\setminus \left(S(F_i)\cup\{y_i\}\right)\right) \cap S(U_i)=\{w\in V(D)\setminus \left(S(F_i)\cup \{y_i\}\right):\: \orarrow{z_w w \in R}\}
\end{align*}
That is, the new sources created in $U_i$ are the vertices from $V\setminus \left(S(F_i)\cup \{y_i\}\right)$ such that the edge preceding them in $F_i$ was removed. Since the elements of $R$ are chosen independently with probability $q$, the variable $g(v)$ follows a binomial distribution $\Bin\left(|N_D^{+}(v)\setminus \left(S(F_i)\cup\{y_i\}\right)|, q\right)$ with average
\begin{align*}
    q|N_D^+(v)\setminus \left(S(F_i)\cup \{y_i\}\right)|&\geq q\left(|N_D^+(v)|-|S(F_i)\cup\{y_i\}|\right) \\&\geq q\left((1-o(1))np-4\log n -1\right)\geq (1-o(1))npq,
\end{align*}
where we use that $N_D^{+}(v)=(1-o(1))np$ (property (D1) of \Cref{typicaldirected}), $|S(F_i)|=m_i\leq 3 \log n$. 
By Chernoff's bound (Theorem \ref{chernoff}), we have that
\begin{align*}
\PP\left(g(v)<\frac{4npq}{5}\right)\leq \exp\left(-\frac{npq}{100}\right)=\exp\left(-\frac{(np)^{1/3}}{200}\right)=o\left(\frac{1}{n^2}\right)
\end{align*}
for $q=t^{-2/3}$ and $t\leq 3np/2$. Therefore, by an union bound and Claim \ref{clm:components}, we obtain that
\begin{align*}
    g(v)\geq \frac{4npq}{5}\geq \frac{3p\ell_i}{4}
\end{align*}
for every vertex $v\in V$ with probability $1-o(1/n)$. Hence, we obtain that whp $\delta^+(H_{U_i,D})\geq 3p\ell_i/4$ for $1\leq i \leq t$. The proof that the event $\delta^-(H_{U_i,D})\geq 3p\ell_i/4$ holds whp is analogous.
\end{proof}

Finally, we show that the graph $R$ has small maximum degree.

\begin{claim}\label{cl:smalldegree}
The event $\Delta(R)\leq 7t^{1/3}$ holds whp.
\end{claim}

\begin{proof}
We claim that whp $\deg_R^+(v)\leq 7t^{1/3}$ for $v\in V$. If $v=x$, then $\deg_R^+(x)=0\leq 7t^{1/3}$ by our construction. Otherwise, the degree $\deg_R^+(v)$ follows a binomial distribution $\Bin(\deg_{\cF}^+(v),q)$ with average $q\deg_{\cF}^+(v)\leq qt=t^{1/3}$ and we obtain that
\begin{align*}
    \PP\left(\deg_R^+(v)>7qt\right)\leq e^{-7qt}=e^{-7t^{1/3}}=o\left(\frac{1}{n}\right)
\end{align*}
by \ref{chernoff2} and $t\geq np/2$. The result now follows from an union bound. The proof that whp $\deg_R^-(v)\leq 7t^{1/3}$ for $v\in V$ is similar. If $v=x$, then $\deg_R^-(x)$ follows a binomial distribution $\Bin(\deg_{\cF}^-(x),q)$. Otherwise, the degree $\deg_R^{-}(v)$ follows $\Bin(\deg_{\cF}^-(v)-1,q)$. The rest of the argument is identical to the argument for $\deg_R^+(v)$. Therefore, whp $\Delta(R)\leq 7t^{1/3}$.
\end{proof}

We finish the proof as follows. Let $R$ and $\{U_i\}_{i=1}^t$ satisfy Claims \ref{clm:components}, \ref{cl:auxgraph} and \ref{cl:smalldegree}. Such a choice exists because all the claims hold with high probability. Note that Claim \ref{clm:components} gives us that $\ell_i\geq (1-o(1))nq=(1-o(1))n/t^{2/3}\geq \frac{n^{1/3}}{2p^{2/3}}$ for $3np/2\geq t\geq np/2$ and $1\leq i\leq t$. This implies that the linear forest $U_i$ satisfies the hypothesis of \Cref{prop:auxgraph}. Therefore, there exists a family of Hamilton cycles $\cC=\{C_1,\ldots,C_t\}$ such that $C_i \in D\cup F_i$ and $\orarrow{xy_i}\in U_i\subseteq C_i$. The condition on the degree of $R$ is immediately satisfied by Claim \ref{cl:smalldegree}.
\end{proof}

We finish the section by showing a way to fully cover matchings by Hamilton cycles. This will be useful for covering the graph $R$ of edges not covered by \cref{lem:keycover}. A digraph $M$ is called a matching if the undirected underlying graph is a matching. First, we prove that we can cover a small matching with a Hamilton cycle in $D_{n,p}$
\begin{proposition}\label{prop:smallmatch}
Let $D:=D_{n,p}$ with $1/2 \geq p\geq \frac{\log^{10} n}{n} $ and let $j\geq 1$ be an integer. Then whp every matching $M$ with at most $j$ edges can be covered by a Hamilton cycle in $D\cup M$.
\end{proposition}

\begin{proof}

We will show that $M$ satisfies the hypothesis of \cref{prop:auxgraph}. First, note that when we consider $M$ as a directed linear forest, it has at least $n- 2j \geq \frac{n^{1/3}}{2p^{2/3}}$ isolated vertices, which are components. Second, by property (D1) of \Cref{typicaldirected}, with high probability, the minimum degree in $D$ is $(1+o(1))np$. In the auxilliary digraph $H_{M,D}$, only $j$ pairs of vertices are contracted. Thus, the minimum degree in $H_{M,D}$ is at least $(1+o(1))np - j > 3(n-j)p/4$. Now, we may apply \cref{prop:auxgraph} to obtain a Hamilton cycle in $D \cup M$ covering $M$.
\end{proof}

The next result shows that any matching in $D_{n,p}$ can be covered with constantly many Hamilton cycles, regardless of size.

Again, by technical reasons to be explained later, the covering is done in a way to include a specific set of extra edges.

\begin{lemma}\label{lem:covermatch}
Let $D:=D_{n,p}$ with $1/2 \geq p\geq \frac{\log^{10} n}{n}$ and $k\geq 10$ be an integer. The following holds whp. Let $M$ be a matching and $x\in V(D)$ a vertex that is not matched in $M$, 
Let $S=\{\orarrow{xy_1},\ldots, \orarrow{xy_k}\}$ be a set of $k$ edges going out from $x$. Then there exists a family of $k$ Hamilton cycles $\cC=\{C_1,\ldots, C_k\}$ in $D\cup M\cup S$ such that $\cC$ covers the edges of $M$ and $\orarrow{xy_i}\in C_i$ for every $1\leq i \leq k$.
\end{lemma}

\begin{proof}
    We will start by partitioning $M$ into $k-1$ smaller matchings $M_1, \ldots M_{\lfloor k/2\rfloor}$ as follows: First, assign each vertex $v$ a number $\phi(v) \in\left\{ 1, 2, \ldots \lfloor k/2\rfloor \right \}$ uniformly at random. Then, for each $i \in \left\{ 1, 2, \ldots \lfloor k/2\rfloor \right \}$, let $M_i$ be the matching containing all edges $\overrightarrow{ab} \in M$ such that $\phi(b)=i$. 
    Now, define $$M_{i}'= \{ \overrightarrow{ab} \mid \overrightarrow{ab} \in M_{j}, \text{  }  b \neq y_i\} \cup \{\overrightarrow{x y_i}\}.$$
    For each $i$, we will use \cref{prop:auxgraph} to find a Hamilton cycle covering $M_i$. First, since each $M_i$ is a matching, and we consider isolated vertices to be components, the number of components $\ell_i$ in each $M_i$ is least $n/2 \geq \frac{n^{1/3}}{2p^{2/3}}$. Thus, it is sufficient to prove:
\begin{claim}\label{cl:auxgraphremainder}
For $1\leq i \leq \lfloor k/2\rfloor$, whp we have that $\delta(H_{M_i,D})\geq 3p\ell_i/4$.    
\end{claim}
\begin{proof}
Let $t=t(P)$ be the sink of a path in $M_{i}'$, and let $S(M_i')$ be the set of sources of paths in $M_i'$. 
By property (D1) in \cref{typicaldirected}, $\deg^+_D(t)=(1+o(1))np$. Now, note that $\deg^+_{H_{M_i',D}}(P)= |N^+_D(t) \cap S(M_i')|$. Each vertex $v \neq x$ of $N^+_D(t)$ is isolated in $M_i'$ if $\phi(v) \neq i$, which occurs with probability at least $4/5$. Let $I(N^+_D(t),M_i')$ be the number of vertices in $N^+_D(t)$ isolated in $M_i'$. The expectation of $I(N^+_D(t), M_i')$ is $np(\lfloor k/2 \rfloor -1)/\lfloor k/2 \rfloor \geq 4np/5$. Thus, by Chernoff's bound (\cref{chernoff}):
$$\Prob\left(I(N^+_D(t), M_i') < 3np/4 \right) \leq \exp\left(\frac{np}{100} \right)= o\left(\frac{1}{n^2}\right).$$
Since each isolated vertex is a source in $M_i'$, 
$$\Prob\left(\deg^+_{H_{M_i',D}}(P)< 3np/4 \right) = o\left(\frac{1}{n^2}\right).$$
We can use the same argument to show $\Prob\left(\deg^-_{H_{M_i',D}}(P)< 3np/4 \right) = o\left(\frac{1}{n^2}\right).$
By a simple union bound, with probability $o\left(\frac{1}{n}\right)$, it holds that $\delta(H_{M_i,D})\geq 3p\ell_i/4$.
\end{proof}
Now, by applying \cref{prop:auxgraph} we obtain Hamilton cycles $C_1 \ldots C_{\lfloor k/2\rfloor}$ covering the matchings $M_1' \ldots M_{\lfloor k/2\rfloor}'$.  
We still need to cover the edges $xy_{\lfloor k/2\rfloor+1} \ldots xy_k$ and the set $Q$ of edges that were contained in some $M_i$ but not contained in $M_i'$.
For each $1 \leq i \leq \lfloor k/2 \rfloor$, if $M_{i}$ contains an edge $\overrightarrow{z_iy_i}$ with $y_i$ as a sink, then $M_i'$ does not contain that edge. In that case, let $F_{i+\lfloor k/2 \rfloor}$ be the linear forest containing $\overrightarrow{z_iy_i}$ and $xy_{i+\lfloor k/2 \rfloor}$. If $M_i$ does not contain an edge with $y_i$ as a sink, let $F_{i+\lfloor k/2 \rfloor}$ be the singleton linear forest containing $xy_{i+\lfloor k/2 \rfloor}$. If $k$ is odd, define $F_{k}$ as the singleton linear forest containing $xy_k$. Each of these linear forest has at most two elements, so by \cref{prop:smallmatch}, there exist Hamilton cycles $C_{\lfloor k/2\rfloor} \ldots C_k$ in $D \cup M \cup S$ covering each $F_i$. This concludes the proof of \cref{lem:covermatch}.
\end{proof}

\section{Proof of Theorem \ref{thm:main}}\label{sec:main}

Finally, we put our work together and prove Theorem \ref{thm:main}

\begin{proof}[Proof of Theorem \ref{thm:main}]
First, note that whp $D:=D_{n,p}$ simultaneously satisfies \Cref{typicaldirected} and all the statements of Sections \ref{sec:coverforest} and \ref{sec:coverhamilton}. We claim that any such $D$ can be covered by $\Delta(D)$ Hamilton cycles. Let $\Delta_1\geq \Delta_2\geq \ldots \geq \Delta_{2n}$ be the ordered sequence of out and in-degrees of $D$. By property (D1) of \Cref{typicaldirected}, we have that
\begin{align}\label{eq:degrelation}
    (1-o(1))np\leq \Delta_{2n}\leq \ldots \leq \Delta_1\leq (1+o(1))np,
\end{align}
and by \Cref{prop:dgap} we also have that
\begin{align}\label{eq:dgap}
\Delta_1-\Delta_2\geq \frac{\sqrt{np}}{2\log n}.
\end{align}
Therefore, the digraph $D$ has a unique vertex of maximum degree. Let $x_\star$ be this vertex and suppose  without loss of generality that $\Delta_1=\deg_D^{+}(x_\star)$. That is, the maximum degree is an out-degree.

\Cref{lem:almostforest} applied to $D$ gives us a family $\cF=\{F_1,\ldots, F_t\}$ of linear forests with $\Delta_2-3\leq t\leq \Delta_2$, and disjoint graphs $R_1$ and $E:=D\setminus\left(R_1\cup\left(\bigcup_{i=1}^t F_i\right)\right)$, such that

\begin{enumerate}
    \item[(i)] For $1\leq i \leq t$, the forest $F_i$ has at most $4\log n$ components.
    \item[(ii)] For $1\leq i \leq t$, there exists vertex $y_i\in V(D)$ such that $\orarrow{x_\star y_i}\in F_i$.
    \item[(iii)] The edge set $E$ consists of $\Delta_1-t$ out-edges from $x_{\star}$.
    \item[(iv)] The maximum degree of $R_1$ is $\Delta(R_1)\leq O\left((np\log^4 n)^{1/3}\right)$. 
\end{enumerate}
By (\ref{eq:degrelation}), we obtain that $\frac{np}{2}\leq t \leq \frac{3np}{2}$, which implies that the family of linear forest $\cF$ satisfies the hypothesis of \Cref{lem:keycover}. Thus, the family $\cF$ can be covered by a family of Hamilton cycles $\cC_0=\{C_1,\ldots, C_t\}$ with $\orarrow{x_\star y_i}\in C_i$ and a graph $R_2$ satisfying
\begin{align}\label{eq:sizeR2}
    \Delta(R_2)\leq 7t^{1/3}\leq 14(np)^{1/3}.
\end{align}

Hence, by using $t$ Hamilton cycles we manage to cover all the edges of $D$ but a set $E$ of out-edges from $x_\star$ and a sparse subgraph $R:=R_1\cup R_2$. By inequalities (\ref{eq:dgap})  and (\ref{eq:sizeR2}), by property (iii) and (iv), and by the assumption $p\geq \frac{\log^{20}n}{n}$, it follows that
\begin{align*}
    \Delta(R)\leq O\left((np\log^4 n)^{1/3}\right) + 14(np)^{1/3}\leq \frac{\sqrt{np}}{\log^2 n}=o\left(\frac{\sqrt{np}}{\log n}\right)=o\left(|E|\right).
\end{align*}

We now describe how to cover the edges of $R$ using Hamilton cycles containing the out-edges in $E$. Consider the undirected underlying multigraph $\tilde{R}$ of $R$, which has the same vertex and edge set as $R$, but the edges are undirected (If $\overrightarrow{ij}$ and $\overrightarrow{ji}$ are in $R$, $\tilde{R}$ contains two edges between $i$ and $j$). By Vizing's theorem (Theorem~\ref{thm:vizing}), one can partition the edges of $\tilde{R}$ into
\begin{align*}
    \Delta(\tilde{R})+2\leq 2\Delta(R)+2\leq 3\sqrt{np}/\log^2 n
\end{align*}
undirected matchings. By reconsidering the orientation, we obtain a family of directed matchings $\cM=\{M_1,\ldots, M_r\}$. Now arbitrarily partition the set of edges $E$ into $r$ almost equidistributed sets $E=\bigcup_{i=1}^r S_i$ with $\big||S_i|-|S_j|\big|\leq 1$ for every $1\leq i, j \leq r$. Note that by (\ref{eq:dgap}) we have
\begin{align*}
    |S_i|=\left\lfloor\frac{|E|}{r}\right\rfloor\geq \frac{\Delta_1-\Delta_2}{3\sqrt{np}/\log^2 n}\geq \frac{\log n}{6}\geq 10
\end{align*}
Therefore, by applying \Cref{lem:covermatch} to each set $S_i\cup M_i$, we obtain a family of $|S_i|$ Hamilton cycles $\cC_i$ covering all the edges in $S_i\cup M_i$. Hence, the family $\bigcup_{i=0}^r\cC_i$ of Hamilton cycles cover all the edges of $D$ with size
\begin{align*}
    \sum_{i=0}^r|\cC_i|=t+\sum_{i=1}^r|S_i|=t+|E|=\Delta_1,
\end{align*}
which concludes the proof of the theorem.
\end{proof}

\bibliographystyle{abbrv}
\bibliography{References}
\appendix
\section{Proof of \cref{prop:ugly}}

Recall the bijection between balanced bipartite graphs and directed graphs (digraphs) determined by a permutation $\pi$. 
For a bipartite graph $B$, with parts $X$ and $Y$, which are two disjoint copies of $[n]$, we write $xy$ for the edge between vertex $x\in X$ and $y\in Y$ (in particular, $xy \neq yx$). Given any permutation $\pi \in S_n$, the digraph $D_{\pi}:=D_{\pi}(B)$ has vertex set $[n]$ and edge set $E\left(D_\pi(B) \right) \coloneqq \left\{\overrightarrow{x \pi(y)} ~:~ xy \in E(B) , x \neq \pi(y)\right\}$.
Let $\mathcal{M}=\{M_1 \ldots M_r\}$ be a collection of edge-disjoint perfect matchings of $B$. For each perfect matching $M_i\in \mathcal M$, the directed graph induced by the edges of $M_i$ under the bijection, denoted by $D_\pi(M_i)$, is a $1$-factor in $D_\pi(B)$. 
For each $M_i$, define $m_i \in S_n$ as the permutation where $m_i(x)=y$ when $xy \in M_i$. Thus, the $1$-factor $D_\pi(M_i)$ is spanned by directed edges of the form $\overrightarrow{x \pi(m_i(x))}$ for $x \in X$.

For a vertex $v \in [n]$ and $1 \leq i \leq r$ we denote by $C_{v,i}$ the (unique) cycle in $D_\pi(M_i)$ that contains $v$, and we let $c_{v,i} \coloneqq |C_{v,i}|$ denote its length. In this section, we will consider isolated vertices to be cycles of length 1.
Note that since $\pi \in S_n$ is a uniformly random permutation, both $C_{v,i}$ and $c_{v,i}$ are random variables.

We restate the proposition proved in this appendix.
\begin{repproposition}{prop:ugly}
	For every $v \in [n]$ we have
	\begin{align*}
	\mathbb E \left(\left(\sum_{i \in [r]} \frac{1}{c_{v,i}} \right)^3 \right) = O\left(\frac{r}{n} \log^3 n \right).
	\end{align*}
\end{repproposition}

Before we prove Proposition \ref{prop:ugly}, we need an auxiliary result that allows us to compute the probability that two or three cycles simultaneously have a prescribed size.

\begin{lemma}\label{lem:highmoments}
    Let $2\leq k\leq 3$ and $1\leq a_1\leq \ldots a_k \leq n/6$. Then for every $v \in [n]$ and collection of $k$ matchings $M_{i_1},\ldots, M_{i_k}$ we have that
\begin{align*}
    \PP\left(\bigcap_{j=1}^k (c_{v,i_{j}}=a_j)\right)=O(n^{-k}).
\end{align*}
\end{lemma}

\begin{proof}
Fix $v \in [n]$. For simplicity, suppose that the collection of $k$ matchings is $M_1, \ldots, M_k$ and write $C_i$ and $c_i$ instead of $C_{v,i}$ and $c_{v,i}$. Throughout the proof we will only present the arguments for the case $k=3$. The corresponding claims for $k=2$ should follow similarly. Let $\cE(a_1,a_2,a_3)= (c_1=a_1)\cap (c_2=a_2) \cap (c_3=a_3)$ be the event that the cycle $c_i$ has length $a_i$ and let $\cE(a_1,a_2)=(c_1=a_1)\cap (c_2=a_2)$ the corresponding event for two cycles. Our goal is to prove that $\PP(\cE(a_1,a_2,a_3))=O(n^{-3})$ and $\PP(\cE(a_1,a_2))$.

 Define the random subset $T:=V\left(C_i\cup C_j\cup C_k\right)\setminus \{v\}$. By a slight abuse of notation, we also consider $T$ to be a subset of $X$. Since the matchings $M_1,M_2$, and $M_3$ are edge-disjoint, we observe that in the bipartite graph $G:=M_1\cup M_2\cup M_3$, the following properties hold:
 \begin{enumerate}
     \item $e_G\left(T,\pi^{-1}(T)\right)\geq c_1+c_2+c_3$, and
     \item the subgraph of $G$ induced by $(T\cup \{v\})\cup\pi^{-1}(T\cup\{v\})$ has no isolated vertices, and
     \item $\pi^{-1}(v)\in Y \textrm{ has three neighbors in } T \cup \{v\}.$
 \end{enumerate}

 These observations are sufficient to prove the following claim:

\begin{claim}\label{clm:cyclesbig}
    $\Prob(|T|=O(1))=O(n^{-3}).$
\end{claim} 

\begin{proof}
    Conditioning on the size of $T$, we distinguish between two cases: 
    \paragraph{\bf{Case 1.}} $v\pi^{-1}(v)\notin E(G)$. In this case, there are $n-3$ possible choices for $y\in Y$, where $y=\pi^{-1}(v)$. The probability that $\pi(y)=v$ is $\frac{1}{n}$. Let $N:=N_G(y)$ denote the set of the three neighbors of $y$ in $G$. Next, choose a subset $S\subseteq X \setminus \left(\{v\}\cup N\right)$ of size $|T|-3$, and define $T:= N\cup S$. 
    
    For each $x\in T$, select a non-empty subset of neighbors $S_x\subseteq N_G(x)$, and let $T':=\cup_{x\in T} S_x\subseteq Y$ be the obtained set. Clearly, there are at most $7^T$ ways to choose $T'$ of size $|T|$ such that, if $\pi(T'\cup \{y\})=T\cup\{v\}$, then its image under $D_{\pi}$ will give us three cycles containing $v$ on $T$. The probability that $\pi(T')=T$ is at most $\frac{T!}{(n)_{T}}$. For $|T|=O(1)$ this is of order $O\left(\left(\frac{1}{n}\right)^{T}\right)$.

Therefore, the probability that the edges $E(T\cup \{v\},\pi^{-1}(T)\cup\{y\})$ map to $T\cup \{v\}$ in $D_{\pi}$, forming three cycles containing $v$, is at most 
$$(n-3)\cdot \frac{1}{n}\cdot \binom{n}{|T|-3}\cdot 7^{|T|}\cdot O(n^{-|T|})=O(n^{-3})$$
as desired. 
\paragraph{\bf{Case 2.}}  $v\pi^{-1}(v)\in E(G)$. This case is quite similar to the previous one, with a few notable differences. Here, there are only $3$ possible choices for $y=\pi^{-1}(v)$, since  $y$ must be chosen as a neighbor of $v$. After selecting $y$ and including its two additional neighbors in $T$, we need to choose a subset $S\subseteq X$ of size $|T|-2$ (instead of $|T|-3$ as in the previous case) extra vertices to complete the set $T$.

Thus, the probability that the edges $E(T\cup \{v\},\pi^{-1}(T)\cup\{y\})$ are mapped to $T\cup \{v\}$ in $D_{\pi}$, forming three cycles containing $v$, is at most 
$$3\cdot \frac{1}{n}\cdot \binom{n}{|T|-2}\cdot 7^{|T|}\cdot O(n^{-|T|})=O(n^{-3})$$
as desired. 

To complete the proof of the claim take an upper bound over all possible values of $|T|$. Since $|T|=O(1)$, the result follows.
\end{proof}

By defining $T_2:=V(C_1\cup C_2)\setminus\{v\}$, a similar argument yields the following bound on the probability that the two shortest cycles have constant length.

\begin{claim}\label{clm:big2}
    $\Prob\left(|T_2|=O(1)\right)=O(n^{-2}).$
\end{claim}

The following claim demonstrates that it is highly unlikely for the three cycles containing $v$ to have significant overlap.

\begin{claim} \label{large intersection}
Let $1\leq a_1\leq a_2\leq a_3$ be such that $\omega(1)=a_1+a_2+a_3\leq n,$. Then, 
    $$\Prob\left((|T|\leq a_1+a_2+a_3/2)\cap \mathcal E(a_1,a_2,a_3) \right)=n^{-\omega(1)}.$$
\end{claim}

\begin{proof}
Suppose that $|T|\leq a_1+a_2+a_3/2$ and $\mathcal E(a_1,a_2,a_3)$ holds. In this case, we must have $e_G(T,\pi^{-1}(T))\geq a_1+a_2+a_3.$ Let $x$ be the number of vertices of degree greater than $1$ in the subgraph of $G$ induced by $T\cup \pi^{-1}(T)$. Now, using the fact that the maximum degree in $G$ is $3$, we obtain that: 
$$a_1+a_2+a_3\leq 3x+a_1+a_2+a_3/2-x,$$
which simplifies to 
$$x\geq a_3/4.$$
Next, note that in this scenario, it is possible to greedily find a subset of $a_3/20=\omega(1)$ vertex-disjoint cherries (pairs of vertices sharing a common neighbor), with the centers of the cherries forming a subset $S\subseteq \pi^{-1}(T)$. This implies that $|N_G(S)|\geq 2|S|$.

Now, consider the number of pairs $(T,T')$ of subsets $T\subseteq X, T'\subseteq Y$ with $|T|=|T'|\leq a_1+a_2+a_3/2$ and there are at least $a_3/20$ vertices in $T'$ whose neighborhood in $T$ is of size at least $a_3/10$. This number is at most 
$$\binom{n}{a_3/20}\binom{n}{|T|-a_3/10}\leq \left(\frac{Cn}{T}\right)^{|T|-a_3/20},$$
for some constant $C>0$. 

Given $T$, there are at most $7^{|T|}$ ways to choose $T'$ (recall that the minimum degree in the graph induced by $T\cup \pi^{-1}(T)$ is at least $1$). Now, for each such pair $(T,T')$, the probability that $\pi(T')=T$ is at most $\binom{n}{|T|}^{-1}\leq (|T|/n)^{|T|}$. Therefore, we obtain that 

$$ \Prob\left((|T|\leq a_1+a_2+a_3/2)\cap \mathcal E(a_1,a_2,a_3) \right)\leq \left(\frac{Cn}{T}\right)^{|T|-a_3/20}\cdot 7^{|T|}\cdot \left(\frac{|T|}{n}\right)^{|T|}=n^{-\omega(1)}. $$

This completes the proof.
\end{proof}

If we consider only the case with two cycles $C_1,$ and $C_2$, a similar proof will give us the following.

\begin{claim}\label{cl:large2} Let $1\leq a_1\leq a_2$ be such that $\omega(1)=a_1+a_2\leq n,$. Then, 
    $$\Prob\left((|T_2|\leq a_1+a_2/2)\cap \mathcal E(a_1,a_2)\right)=n^{-\omega(1)}.$$
\end{claim}

It will be helpful to describe the following exposure procedure for revealing the values of the (random) permutation $\pi: [n] \to [n]$. This procedure is designed to track and extend the directed path in the directed graph $D_{\pi}$ that contains $v$, and ultimately reveal the entire structure of the cycles $C_1$, and then $C_2$ and $C_3$. Given a set of labeled vertices $U\subseteq Y$ and a matching $M_i$, let $M_i[X\cup U]$ be the induced subgraph of $M_i$ on the vertex set $X\cup U$ (where we view $U$ as a subset of $Y$). We define $D[M_i,U]:=D_\pi(M[X\cup U])$. We proceed by exposing the labels of the vertices of $Y$ one step at the time as follows:

\begin{enumerate}
    \item Initialize by setting $U=\emptyset$ (the set of labeled vertices).
    \item For $1\leq i \leq 3$, do the following:
    \begin{enumerate}
        \item[(i)] Let $P$ be the directed component containing the vertex $v$ in $D[M_i,U]$.
        \item[(ii)] While $P$ is not a directed cycle, do the following step: Note that the component $P$ is either a directed path or an isolated vertex. Let $x$ be the sink of $P$ and let $y \in Y$ be the vertex adjacent to $x$ in $M_i$. Since $x$ is a sink, the value of $\pi(y)$ was never revealed. We reveal $\pi(y)$, update $U:=U\cup\{y\}$ and $P$, and repeat.
        \item[(iii)] If $P$ is a directed cycle, we move to the next matching ($i:=i+1$).
    \end{enumerate}
    \item Expose the remaining labels in $[n]\setminus U$ in arbitrary order.
\end{enumerate}

This procedure incrementally reveals the permutation $\pi$ by following the edges of the matchings $M_1,M_2$, and $M_3$
  that define the cycles $C_1,C_2,$ and $C_3$, repectively, in $D_{\pi}$. The key idea is to start at the vertex $v$, trace its unique path in $D_{\pi}(M_{s})$ for $s\in \{i,j,k\}$, and extend the corresponding directed path by uncovering the connections dictated by matchings. For each cycle, we reveal the labels of vertices one step at a time until the entire cycle structure is determined. Once all relevant cycles are exposed, the remaining vertices are labeled arbitrarily. 

Let us now introduce some useful notation. Denote by $t_1,t_2,$ and $t_3$, the closing times of the cycles $C_1,C_2,$ and $C_3,$ respectively, according to the exposure procedure described above. Notice that we always have $t_1=c_1\leq t_2,t_3$, although  it is possible that $t_3<t_2$. At the closing time of a cycle $C_s$, where $s\in [3]$, if we examine the (unique) directed path containing $v$ in the image of the matching $M_s$ in $D_{\pi}$, then the following hold: 
\begin{enumerate}
    \item the path is of length $c_s-1$, and
    \item there is exactly one vertex $w$ (the starting point of this path) that remains unlabeled at this stage; that is, $\pi^{-1}(w)$ is yet undefined. 
\end{enumerate}
 Note that if $z$ is the other endpoint of the path, and we consider $z\in X$, then in order to close the cycle, its unique neighbor in $M_s$, if yet unlabeled, must be labeled $w$. At this point, if there are still $m$ unlabeled vertices in $Y$, the conditional probability that $\pi(z)=w$ is $\frac{1}{m}$.

This observation can be slightly generalized to provide an upper bound on the probability that the three cycles $C_1,C_2$ and $C_3$ have distinct closing times.

\begin{claim}\label{clm:difclosing}
   Let $a_1\leq a_2\leq a_3$ be such that $\omega(1)=a_1+a_2+a_3\leq n/2$. Let $\mathcal A=(t_1 <t_2<t_3)$ be the event that the three cycles close sequentially at distinct times. Then,
   $$\Prob\left(\mathcal E(a_1,a_2,a_3)\cap \mathcal A\right)=O(n^{-3}).$$
\end{claim}

\begin{proof}
Let $\cE^{(3)}:=\cE(a_1,a_2,a_3)\cap \cA$, $\cE^{(2)}:=\cE(a_1,a_2)\cap (t_1<t_2)$ and $\cE^{(1)}:=(c_1=a_1)$. Fix an instance of $\cE^{(3)}$. For $1\leq i\leq 3$, let $U_i$ be the ordered set of labels revealed from steps $t_{i-1}+1$ to $t_i-1$ (where $t_0=0$). Note by the description of the procedure and the observation preceding the claim, that such sets $U_1\cup U_2\cup U_3$ uniquely determines the instance of $\cE^{(3)}$.

Let $P_3$ be the path containing $v$ using the edges of $M_3$ during the procedure. Note that at each step of the process after $t_{2}$, the length of $P_3$ increases in at least one edge. Thus, conditioned on fixed $U_1$ and $U_2$, all the events leading to the possible sequence of vertices $U_{3}$ until step $t_{3}-1$ are disjoint. Moreover, as discussed earlier, the probability of closing the cycles at step $t_3$ conditioned on $U_3$ is at most $2/n$ (since, by assumption, there are always at least $n/2$ unlabeled vertices at any given stage of the procedure). Therefore, 
\begin{align*}
    \PP(\cE^{(3)}\mid U_1,U_2)=\sum_{U_3}\PP(\cE^{(3)}\mid U_1,U_2,U_3)\PP(U_3)\leq\frac{2}{n}\sum_{U_3}\PP(U_3)=\frac{2}{n}.
\end{align*}
Consequently, we obtain that
\begin{align*}
    \PP(\cE^{(3)})=\sum_{U_1,U_2}\PP(\cE(a_1,a_2,a_3)\mid U_1,U_2)\PP(U_1, U_2)=\frac{2}{n}\sum_{U_1,U_2}\PP(U_1, U_2)\leq\frac{2}{n}\PP(\cE^{(2)}).
\end{align*}
A similar argument shows that
\begin{align*}
    \PP(\cE^{(2)})\leq \frac{2}{n}\PP(\cE^{(1)}).
\end{align*}
Since $\PP(\cE^{(1)})=\frac{1}{n}$, by putting together the two inequalities, we have that $\PP(\cE^{(3)})=O(n^{-3})$. This concludes the proof.
\end{proof}


    

The last piece of the puzzle that we need is the following claim, which basically combines all the above: 

\begin{claim}
    Let $1 \leq a_1\leq a_2 \leq a_3\leq n/6$. Then $$\Prob\left(\mathcal E(a_1,a_2,a_3)\right)=O(n^{-3}).$$
\end{claim}

\begin{proof}
    We compute the probability by splitting into several cases depending on the times that each cycle close. Indeed, note that since $t_1\leq t_2,t_3$, the event $\cE(a_1,a_2,a_3)$ can be written as the disjoint union of the events
    \begin{align*}
        \cE(a_1,a_2,a_3)=\cE_1\cup \cE_2 \cup \cE_3 
    \end{align*}
    where we define the events $\cE_i$ by
    \begin{align*}
        \cE_1&:=\cE(a_1,a_2,a_3)\cap (t_3\leq t_2)\\
        \cE_2&:=\cE(a_1,a_2,a_3)\cap  (t_1=t_2<t_3)\\
        \cE_3&:=\cE(a_1,a_2,a_3)\cap (t_1 < t_2 < t_3).
    \end{align*}
    We bound $\PP(\cE_i)$ for $1\leq i \leq 3$.
    \paragraph{\textbf{Case 1: $\cE_1$}} Let $S=V\left(C_1\cup C_2\cup C_3\right)$. Note that in this case $|S|\leq a_1+a_2<a_1+a_2+a_3/2$. Hence, by Claim \ref{clm:cyclesbig} and \ref{large intersection}, we obtain that
    \begin{align*}
        \PP(\cE_1)=\PP(\cE_1\cap (|S|=O(1)))+\PP(\cE_1\cap (|S|=\omega(1)))=O(n^{-3})+O(n^{-\omega(1)})=O(n^{-3}).
    \end{align*}

    \paragraph{\textbf{Case 2: $\cE_2$}} Let $S=V\left(C_1\cup C_2\right)$. Note that in this case $|S|=a_1\leq a_1+a_2/2$. Let $\tilde{\cE}_2$ be the event $\cE(a_1,a_2)\cap (t_1=t_2)$. Thus, by Claim \ref{clm:big2} and \ref{cl:large2}, we obtain that
    \begin{align*}
        \PP(\tilde{\cE_2})=\PP(\tilde{\cE_2}\cap (|S|=O(1)))+\PP(\tilde{\cE_2}\cap (|S|=\omega(1)))=O(n^{-2}).
    \end{align*}
    Let $U$ be the ordered set of labels revealed from steps $t_2$ to step $t_3-1$ conditioned on a instance $\omega \in \tilde{\cE_2}$. By a similar argument as in Claim \ref{clm:difclosing}, we obtain that
    \begin{align*}
        \PP(\cE_2\mid \omega)=\sum_{U}\PP(\cE_2\mid \omega, U)\PP(U)\leq \frac{2}{n}. 
    \end{align*}
    Therefore,
    \begin{align*}
        \PP(\cE_2)=\sum_{\omega \in \tilde{\cE_2}}\PP(\cE_2\mid \omega)\PP(\omega)\leq \frac{2}{n}\PP(\tilde{\cE_2})=O(n^{-3}).
    \end{align*}

    \paragraph{\textbf{Case 3: $\cE_3$}}
    We obtain that $\PP(\cE_3)=O(n^{-3})$ immediately from Claim \ref{clm:difclosing}.

    By putting all the cases together, we obtain that 
    \begin{align*}
    \PP(\cE(a_1,a_2,a_3))=O(n^{-3})
    \end{align*}
    as desired.
\end{proof}

We note that a similar bound can be obtained for two cycles. More precisely, one can prove that $\PP(\cE(a_1,a_2))=O(n^{-2})$. This finishes the proof of the lemma.
\end{proof}

We are now able to prove Proposition \ref{prop:ugly}.

\begin{proof}[Proof of Propositon \ref{prop:ugly}]
Let $v \in [n]$. For simplicity, we will write $C_i$ and $c_i$ instead of $C_{v,i}$ and $c_{v,i}$, respectively. We begin with an auxiliary result.

\begin{claim}\label{clm:expect}
The following three bound holds:
\begin{enumerate}
		\item For every $i \in [r]$, 
		$\mathbb E\left(\frac{1}{c^3_{i}} \right) = O\left(\frac{1}{n} \right)$.
		\item For every distinct $i,j \in [r]$,
		$\mathbb E \left(\frac{1}{c^2_{i} c_{j}} \right) = O\left(\frac{\log^2 n}{n^2} \right).$
		
		\item For every distinct $i,j,k \in [r]$,
		$\mathbb E \left(\frac{1}{c_{i} c_{j} c_{k}} \right) = O \left(\frac{\log^3 n}{n^3} \right)$.
\end{enumerate}
\end{claim}

\begin{proof}
We start by showing statement (1). Note that for a single cycle, the probability $\PP(c_i=t)=1/n$. Hence,
\begin{align*}
    \EE\left(\frac{1}{c_i^3}\right)=\sum_{a_i\in [n]}\frac{1}{a_i^3}\PP(c_i=a_i)=\frac{1}{n}\sum_{a_i\in [n]}\frac{1}{a_i^3}=O\left(\frac{1}{n}\right).
\end{align*}

To prove statement (2) and (3) we use Lemma \ref{lem:highmoments}. Indeed, for statement (2) we have that
\begin{align*}
    \EE\left(\frac{1}{c_i^2c_j}\right)&=\sum_{a_i,a_j \in [n]}\frac{1}{a_i^2a_j}\PP\left((c_i=a_i)\cap(c_j=a_j)\right)\\ &\leq \sum_{a_i,a_j \in [n]}\frac{1}{a_ia_j}\PP\left((c_i=a_i)\cap(c_j=a_j)\right)\leq \sum_{a_i\leq a_j}\frac{2}{a_ia_j}\PP\left((c_i=a_i)\cap(c_j=a_j)\right)\\
    &\leq\sum_{a_i\leq a_j\leq n/4}\frac{2}{a_ia_j}\PP\left((c_i=a_i)\cap(c_j=a_j)\right)+\sum_{a_i\in [n]}\frac{8}{na_i}\PP\left(c_i=a_i\right)\\
    &=O(n^{-2})\left(\sum_{a_i\leq a_j\leq n/4}\frac{1}{a_ia_j}+\sum_{a_i\in [n]}\frac{1}{a_i}\right)=O\left(\frac{\log^2 n}{n^2}\right),
\end{align*}
where we use Lemma \ref{lem:highmoments} for two cycles.

A similar computation holds for statement (3):
\begin{align*}
    \EE\left(\frac{1}{c_ic_jc_k}\right)&=\sum_{a_i,a_j,a_k \in [n]}\frac{1}{a_ia_ja_k}\PP\left((c_i=a_i)\cap(c_j=a_j)\cap(c_k=a_k)\right)\\
    &\leq \sum_{a_i\leq a_j\leq a_k}\frac{6}{a_ia_ja_k}\PP\left((c_i=a_i)\cap(c_j=a_j)\cap(c_k=a_k)\right)
    =\Sigma_1+\Sigma_2+\Sigma_3,
\end{align*}
where
\begin{align*}
    &\Sigma_1:=\sum_{a_i\leq a_j\leq a_k\leq n/6}\frac{6}{a_ia_ja_k}\PP\left((c_i=a_i)\cap(c_j=a_j)\cap(c_k=a_k)\right)\\
    &\Sigma_2:= \sum_{\substack{a_i\leq a_j\leq n/6\\ a_k>n/6}}\frac{6}{a_ia_ja_k}\PP\left((c_i=a_i)\cap(c_j=a_j)\cap(c_k=a_k)\right), \textrm{ and }\\
    &\Sigma_3:=\sum_{\substack{a_i\in [n]\\ n/6<a_j\leq a_k}}\frac{6}{a_ia_ja_k}\PP\left((c_i=a_i)\cap(c_j=a_j)\cap(c_k=a_k)\right)
\end{align*}

We bound each of the $\Sigma_i$s separately as follows: 

By Lemma \ref{lem:highmoments} applied for three cycles, we know that for $a_1\leq a_2\leq a_3\leq n/6$, the probability $\PP((c_i=a_i)\cap (c_j=a_j)\cap (c_k=a_k))= O(n^{-3})$. Therefore,
\begin{align*}
    \Sigma_1=O\left(n^{-3}\sum_{a_1\leq a_2\leq a_3\leq n/6}\frac{1}{a_1a_2a_3}\right)=O\left(\frac{\log^3 n}{n^3}\right)
\end{align*}
Similarly, Lemma \ref{lem:highmoments} applied for two cycles gives that if $a_1\leq a_2\leq n/6$, then $\PP((c_i=a_i)\cap(c_j=a_j))=O(n^{-2})$. Hence
\begin{align*}
    \Sigma_2\leq \sum_{a_i\leq a_j\leq n/6}\frac{36}{na_ia_j}\PP\left((c_i=a_i)\cap(c_j=a_j)\right)= O\left(n^{-3}\sum_{a_i\leq a_j\leq n/6}\frac{1}{a_ia_j}\right)=O\left(\frac{\log^2 n}{n^3}\right)
\end{align*}
Finally, to bound $\Sigma_3$, we just need to use that $\PP(c_i=a_i)=1/n$. Therefore,
\begin{align*}
\Sigma_3 \leq \sum_{a_i\in [n]}\frac{216}{n^2a_i}\PP\left(c_i=a_i\right)=\frac{216}{n^3}\sum_{a_i\in [n]}\frac{1}{a_i}=O\left(\frac{\log n}{n^3}\right)
\end{align*}
By summing everything together, we obtain 
\begin{align*}
    \EE\left(\frac{1}{c_ic_jc_k}\right)=O\left(\frac{\log^3 n}{n^3}\right).
\end{align*}
This concludes the proof of the claim.
\end{proof}

We finish the proof of the proposition by applying Claim \ref{clm:expect} after expanding the third moment
\begin{align*}
	\mathbb E \left(\left(\sum_{i \in [r]}\frac{1}{c_{i}} \right)^3 \right) &= \sum_{i \in [r]} \mathbb E\left(\frac{1}{c_{i}^3} \right) + 3\sum_{\substack{i, j \in [r] \\ i \neq j}} \mathbb E \left(\frac{1}{c_{i}^2 c_{j}} \right) + 6\sum_{\substack{i,j,k \in [r] \\ i<j<k}} \mathbb E \left(\frac{1}{c_{i} c_{j} c_{k}} \right)\\
  &=r \cdot O\left(\frac{1}{n} \right) + r^2 \cdot O\left(\frac{\log^2 n}{n^2} \right) + r^3 \cdot O\left(\frac{\log^3 n}{n^3} \right) = O\left(\frac{r}{n} \log^3 n \right),
 \end{align*}
 as desired.
\end{proof}

\end{document}